\documentclass[reqno,centertags,11pt,draft]{amsart}
\usepackage{amsmath,amsthm,amsfonts,amssymb,enumerate}
\usepackage[bookmarksopen=true,final]{hyperref}
\usepackage{mathrsfs}
\usepackage[nobysame]{amsrefs}
\usepackage{hyperref}
\usepackage[pdftex]{graphicx,color}
\usepackage{graphicx}
\usepackage{marvosym}
\usepackage{scalerel}
\usepackage{bbm}
\usepackage{bm}
\usepackage{setspace}
\usepackage{tikz}
\usepackage[foot]{amsaddr}
\usepackage[normalem]{ulem}
\usepackage{cases}

\usepackage{fancyhdr} 
\usepackage{ulem} 

\usepackage{arydshln}

\usepackage{xcolor}

\usepackage[a4paper,tmargin=4cm,bmargin=4.5cm,lmargin=3.2cm,rmargin=3.2cm]{geometry}

\newtheorem{thm}{Theorem}[section]

\newtheorem{cor}[thm]{Corollary}
\newtheorem{prop}[thm]{Proposition}
\newtheorem{rem}[thm]{Remark}

\newcommand{\R}{\mathbb{R}}
\newcommand{\C}{\mathbb C}

\newcommand{\Z}{\mathbb{Z}}
\newcommand{\N}{\mathbb{N}}
\newcommand{\T}{\mathbb{\partial\mathbb{D}}}
\newcommand{\D}{\mathbb{D}}

\newcommand{\abs}[1]{\left|#1\right|}
\newcommand{\set}[1]{\left\{#1\right\}}

\begin{document}
\title{
Szeg\H{o} mapping and Hermite--Pad\'{e} polynomials for multiple orthogonality on the unit circle
}
\date{\today}
\author{Rostyslav Kozhan$^1$}
\email{kozhan@math.uu.se}
\author{Marcus Vaktnäs$^{1}$}
\email{marcus.vaktnas@math.uu.se}
\address{$^{1}$Department of Mathematics, Uppsala University, S-751 06 Uppsala, Sweden}
\begin{abstract}
We investigate generalized Laurent multiple orthogonal polynomials on the unit circle satisfying simultaneous orthogonality conditions with respect to $r$ probability measures or linear functionals on the unit circle. We show that these polynomials can be characterized as solutions of a general two-point Hermite--Padé approximation problem.

We derive Szeg\H{o}-type recurrence relations, establish compatibility conditions for the associated recurrence coefficients, and obtain Christoffel--Darboux formulas as well as Heine-type determinantal representations.

Furthermore, by extending the Szeg\H{o} mapping and the Geronimus relations, we relate these Laurent multiple orthogonal polynomials to multiple orthogonal polynomials on the real line, thereby making explicit the connection between multiple orthogonality on the unit circle and on the real line.



\end{abstract}
\maketitle

\section{Introduction}

While multiple orthogonality on the real line has been extensively developed over the past decades, its analogue on the unit circle, introduced by M\'{i}nguez and Van Assche~\cite{MOPUC1}, still remains a relatively new area of study. Nevertheless, recent work~\cite{MOPUC2,KVMOPUC,KVMLOPUC1,KNik,HueMan} indicates that multiple orthogonality on the unit circle exhibits a rich and intricate structure that warrants a unified and systematic treatment.

In this paper, we introduce a generalized notion of Laurent multiple orthogonal polynomials, which provides a natural framework for the generalized Hermite--Pad\'{e} problem and for connecting with the Szeg\H{o} mapping and Geronimus relations, linking the theory to multiple orthogonal polynomials on the real line. Before presenting our main results, we summarize some foundational aspects of the theory that will be needed throughout the paper. In particular, we summarize key facts about orthogonal polynomials on the real line (Section~\ref{ss:intro1}) and on the unit circle  (Section~\ref{ss:intro2}), their connections to Pad\'e and Hermite--Pad\'e approximation, and the Szeg\H{o} mapping  (Section~\ref{ss:intro3}). A concise overview of our main results is provided in Section~\ref{ss:main}.






\subsection{Orthogonality on the real line and Hermite--Pad\'{e} approximation}\label{ss:intro1}
\hfill\\

We start by reviewing the basic theory of orthogonal and multiple orthogonal polynomials on the real line, together with the associated Pad\'{e} and Hermite--Pad\'{e} approximation problems. 
Let $\mu$ be a probability measure supported on the real line $\R$ with all  moments 
$$
c_j = \int_\R x^j d\mu(x), \qquad j\in\N=\{0,1,2,\ldots\},
$$
finite. The Cauchy--Stieltjes transform of $\mu$ is the function, analytic on $\C\setminus\R$, given by
\begin{equation}
    \label{eq:m}
m(z) = \int_{\R} \frac{d\mu(x)}{x-z} = -\sum_{j=0}^\infty c_j z^{-j-1}, \qquad z\to\infty. 
\end{equation}
The monic (i.e., with leading term 1) orthogonal polynomials 
of $\mu$ are then defined by requiring $\deg{P_n} = n$ and
\begin{equation}\label{eq:oprl1}
    \int_\R P_n(x)x^{k} d\mu(x) = 0, \qquad k = 0,\dots,n-1.
\end{equation}
They satisfy the three-term recurrence relation
\begin{equation}
    x P_{n}(x) = P_{n+1}(x) + b_n P_n(x) + a_n P_{n-1}(x),
\end{equation}
where $a_n>0$, $n\ge 1$, ($a_0:=0$) and $b_n\in\R$ are called the Jacobi coefficients.
See~\cite{SzegoBook,SimonL2} for more details.

One of the classical settings in which orthogonal polynomials naturally appear is Pad\'{e} approximation, see, e.g.,~\cite{Baker} for a comprehensive treatment.
Given a function $m(z)$ as in~\eqref{eq:m}, the Pad\'{e} problem seeks polynomials $P_n$ and $Q_n$ of degrees at most $n$ and $n-1$, respectively, such that
\begin{equation}\label{eq:padeR}
P_n(z) m(z) +Q_n(z) = \mathcal{O}(z^{-n-1}), \qquad z \to \infty.
\end{equation}
It can be shown that all the 
solutions $P_n$ of 
this problem are exactly the orthogonal polynomials defined by~\eqref{eq:oprl1}.

The Hermite--Pad\'{e} approximation problem concerns simultaneous approximation of Cauchy--Stieltjes transforms, generalizing~\eqref{eq:padeR}. 
Let $m_1(z), \ldots, m_r(z)$ be the Cauchy--Stieltjes transforms of $r$ measures $\mu_1, \ldots, \mu_r$ on $\R$. For a multi-index $\bm{n} = (n_1, \ldots, n_r) \in \N^r$, we look for polynomials $P_{\bm{n}}$ and $Q_{\bm{n},j}$  of degree $\le |\bm{n}|:=n_1+\ldots+n_r$ and $\le |\bm{n}|-1$ satisfying
\begin{equation}\label{eq:hermitePadeR}
P_{\bm{n}}(z) m_j(z) + Q_{\bm{n},j}(z) = \mathcal{O}(z^{-n_j - 1}),
\qquad z \to \infty, \quad j = 1,\ldots,r.
\end{equation}

One of the fundamental results in this theory is that a polynomial $P_{\bm{n}}$ 
solves~\eqref{eq:hermitePadeR} if and only if it satisfies simultaneous orthogonality conditions 
\begin{equation}\label{eq:MOPRL}
    \int_\R P_{\bm{n}}(x) x^{k}  d\mu_j(x) = 0,
    \qquad k = 0,\ldots,n_j-1, \quad j = 1,\ldots,r,
\end{equation}
with respect to the system of measures 
$\bm{\mu}=(\mu_1,\ldots,\mu_r)$. Such polynomials are called (type II) multiple orthogonal polynomials of $\bm\mu$.
This is a very well developed area of research, see, e.g.,~\cites{Aptekarev, Ismail, Applications,Nikishin} and references therein.

\subsection{Orthogonality on the unit circle and two-point Hermite--Pad\'{e} approximation}\label{ss:intro2}
\hfill\\


Let us now introduce the parallel theory on the complex unit circle~\cites{OPUC1,OPUC2,GeronimusBook,SzegoBook}. We start with  a probability measure supported on  $\T=\{z\in\C: |z|=1\}$, and denote its moments 
$$
c_j = \int_\T w^{-j} d\mu(w), \qquad j\in\Z.
$$
The Carath\'{e}odory function of $\mu$ is defined to be the analytic on $\C\setminus\T$ function $F(z)$ given by
\begin{equation}\label{eq:F}
    F(z) = \int_\T \frac{w+z}{w-z} d\mu(w)
    =
    \begin{cases}
         c_{0} + 2\sum_{k = 1}^{\infty}c_{k}z^k, & z\to0,
         \\
        -c_{0} - 2\sum_{k = 1}^{\infty}c_{-k}z^{-k}, & z\to\infty.
    \end{cases}
\end{equation}
The monic orthogonal polynomials $(\Phi_n(z))_{n = 0}^\infty$ of $\mu$ are defined by requiring $\deg{\Phi_n} = n$ and
\begin{equation}\label{eq:opuc1}
    \int_\T \Phi_n(w)w^{-k} d\mu(w) = 0, \qquad k = 0,\dots,n-1.
\end{equation}
They satisfy the Szeg\H{o} recurrence relations given by
\begin{align}
\label{eq:scalar szego 1}
\Phi_{n+1}(z) & = z\Phi_n(z) + {\alpha}_{n+1}\Phi^*_{n}(z),
\\
\label{eq:scalar szego 2}
\Phi^*_{n+1}(z) & =  \Phi^*_n(z) + \bar\alpha_{n+1}z\Phi_{n}(z),
\end{align}
where the Szeg\H{o} dual polynomials $\Phi_n^*(z) = z^n\overline{\Phi_n(1/\bar z)}$ can be defined via $\Phi_n^*(0)=1$ and
\begin{equation}\label{eq:Phi^*}
    \int_\T \Phi^*_n(w)w^{-k} d\mu(w) = 0, \qquad k = 1,\dots,n.
\end{equation}
The recurrence coefficients $\alpha_n$, $n\ge 1$, belong to open unit disc $\D = \set{z \in \C : \abs{z} < 1}$, and are known as the Verblunsky coefficients of $\mu$ (also referred to in the literature as Schur, Szeg\H{o}, or Geronimus coefficients).\footnote{We warn the reader that here we are choosing $\alpha_n = \Phi_n(0)$ instead of the nowadays more commonly used $\alpha_n = \overline{\Phi_{n+1}(0)}$ 
(see the historical discussion in~\cite[p.10]{OPUC1}).}

The associated Pad\'{e} approximation problem takes the following {\it two-point} form, see, e.g.~\cite{Baker,Bultheel,JNT,PeherstorferSteinbauer} and references therein.
Given a function $F(z)$ as in~\eqref{eq:F}, we seek polynomials $\Phi_n$ and $\Psi_n$ of degrees at most $n$, such that
\begin{alignat}{3}
\label{eq:padeT1}
    & \Phi_{{n}}(z)F(z) + \Psi_{{n}}(z) = \mathcal{O}(z^{n}),  \qquad&& z \rightarrow 0, 
    \\
    \label{eq:padeT2}
    & \Phi_{{n}}(z)F(z) + \Psi_{{n}}(z) = \mathcal{O}(z^{-1}), \qquad && z \rightarrow \infty.
\end{alignat}
All polynomial solutions $\Phi_n$ of~\eqref{eq:padeT1}--\eqref{eq:padeT2} 
are constant multiples of the orthogonal polynomials defined by~\eqref{eq:opuc1}.

In~\cite{MOPUC1}, M\'{i}nguez and Van Assche considered the Hermite--Pad\'{e} approximation problem of finding 
polynomials $\Phi_{\bm{n}}(z),\Psi_{\bm{n},1}(z),\ldots,\Psi_{\bm{n},r}(z)$ of degree $\le|\bm{n}|$ that satisfy
\begin{alignat}{3}
    \label{eq:hermitePadeT1}
    & \Phi_{\bm{n}}(z)F_j(z) + \Psi_{\bm{n},j}(z) = \mathcal{O}(z^{n_j}), \qquad  && z \rightarrow 0, 
    \\
    \label{eq:hermitePadeT2}
    & \Phi_{\bm{n}}(z)F_j(z) + \Psi_{\bm{n},j}(z) = \mathcal{O}(z^{-1}), \qquad  &&z \rightarrow \infty,
\end{alignat}
for $r$ Carath\'{e}odory functions~\eqref{eq:F} $F_1,\ldots,F_r$ associated with measures $\mu_1,\ldots,\mu_r$  on $\T$, and showed that $\Phi_{\bm{n}}(z)$ 
solves~\eqref{eq:hermitePadeT1}--\eqref{eq:hermitePadeT2} if and only if it
fulfills the simultaneous orthogonality relations
\begin{equation}\label{eq:MOPUC}
    \int_\T \Phi_{\bm{n}}(w) w^{-k}  d\mu_j(w) = 0,
    \qquad k = 0,\ldots,n_j-1, \quad j = 1,\ldots,r.
\end{equation}
The polynomials $\Phi_{\bm{n}}(z)$ are therefore  called (type II) multiple orthogonal polynomials on the unit circle. See~\cites{MOPUC1,MOPUC2,KVMOPUC} for further results on the properties of $\Phi_{\bm{n}}$. 

\subsection{Szeg\H{o} mapping and Geronimus relations}\label{ss:intro3}
\hfill\\

Finally, we recall the Szeg\H{o} mapping and the Geronimus relations, which provide a fundamental connection between the theories of orthogonal polynomials on the real line and on the unit circle. Given any probability measure $\gamma$ supported on the real interval $[-2,2]$, one can define a  probability measure $\mu=\operatorname{Sz}(\gamma)$ to be the probability measure on the unit circle $\T$, that is invariant under the reflection $e^{i\theta}\mapsto e^{-i\theta}$ and satisfies
\begin{equation}\label{eq:szego mapIntro}
    \int_{\T} g(2\cos\theta) d\mu(e^{i\theta}) = \int_{-2}^2 g(x) d\gamma(x),
\end{equation}
for all measurable functions $g$ on $[-2,2]$. Conversely,  any probability measure $\mu$ on $\T$ that is invariant under $e^{i\theta}\mapsto e^{-i\theta}$ arises as   $\mu=\operatorname{Sz}(\gamma)$ for a unique probability measure $\gamma$ on $[-2,2]$  determined by~\eqref{eq:szego mapIntro}. Szeg\H{o} showed~\cite{Szego} that the orthogonal polynomials $P_n$ on the real line with respect to $\gamma$ and the orthogonal polynomials $\Phi_n$ on the unit circle with respect to $\mu$ are connected by the identities
\begin{equation}
     P_n(z+z^{-1}) = \frac{1}{1+\alpha_{2n}} \big( z^{-n} \Phi_{2n}(z) + z^{n} \Phi_{2n}(1/z) \big) 
\end{equation}
Later, Geronimus~\cite{GeronimusRelations,GeronimusBook} derived explicit formulas relating the recurrence coefficients of $\gamma$ and  of $\mu$:
\begin{alignat}{3}
\label{eq:GerIntro1}
    a^2_n &  = (1+\alpha_{2n-2})(1-\alpha_{2n-1}^2)(1-\alpha_{2n}),
    \qquad && n\ge 1,
    \\
\label{eq:GerIntro2}
    b_n & =(1-\alpha_{2n})\alpha_{2n-1} - (1+\alpha_{2n})\alpha_{2n+1}, \qquad && n\ge 0,
\end{alignat}
with the convention $\alpha_{-1} = 0$. 

\subsection{Summary of main results}\label{ss:main}
\hfill\\

In this paper we introduce the Laurent multiple orthogonal polynomials defined by
\begin{align}
\label{eq:generalized hermite pade orthogonality intro}
    & \int_\T \Phi_{\bm{n};\bm{m}}(w)w^{-k}d\mu_j(w) = 0, \qquad k = -m_j,\dots,n_j-1, \qquad j = 1,\dots,r,
    \\
    \label{eq:IntroSpanNEW}
    & \Phi_{\bm{n};\bm{m}} \in \operatorname{span}\set{z^{-\abs{\bm{m}}},z^{-\abs{\bm{m}}+1},\dots,{z^{\abs{\bm{n}}}}}.
\end{align}
Here $\bm{n}$ and $\bm{m}$ are two arbitrary multi-indices in $\N^r$. In particular, if $\bm{m}=\bm{0}$, then $\Phi_{\bm{n};\bm{0}}$ are the 
multiple orthogonal polynomials $\Phi_{\bm{n}}$ of M\'{i}nguez--Van Assche~\cite{MOPUC1}.

The motivation for studying this class of polynomials is two-fold. First, we show (see Sections~\ref{ss:general type II hermite pade}--\ref{ss:HPI}) that a Laurent polynomial $\Phi_{\bm{n};\bm{m}}(z)$ satisfying~\eqref{eq:IntroSpanNEW} fulfills the orthogonality conditions~\eqref{eq:generalized hermite pade orthogonality intro} if and only if it solves the two-point Hermite--Pad\'{e} approximation problem 
\begin{alignat}{3}
    \label{eq:intro generalized two point hermite pade eq 1}
    & \Phi_{\bm{n};\bm{m}}(z)F_j(z) + \Psi_{\bm{n};\bm{m},j}(z) = \mathcal{O}(z^{n_j}),  \qquad &&z \rightarrow 0, \\
    \label{eq:intro generalized two point hermite pade eq 2}
    & \Phi_{\bm{n};\bm{m}}(z)F_j(z) + \Psi_{\bm{n};\bm{m},j}(z) = \mathcal{O}(z^{-m_j-1}), \qquad &&z \rightarrow \infty,
\end{alignat}
for suitable Laurent polynomials $\Psi_{\bm{n};\bm{m},j}$. Here $F_j$ are the Carath\'{e}odory functions associated with the measures $\mu_1,\ldots,\mu_r$.
In the scalar case $r=1$, one readily verifies that $\Phi_{n;m}(z) = z^{-m} \Phi_{n+m}(z)$.
If $r\ge 2$, such a trivial relationship with $\Phi_{\bm{n}}$ in~\eqref{eq:MOPUC}, of course, no longer holds. 

Our second motivation for introducing the Laurent framework~\eqref{eq:generalized hermite pade orthogonality intro} is that it naturally accommodates the Szeg\H{o} mapping and the Geronimus relations. This allows us to connect our results with the theory of multiple orthogonal polynomials on the real line, a connection that is not available in the multiple polynomial setting~\eqref{eq:MOPUC} considered in~\cite{MOPUC1,MOPUC2,KVMOPUC}. These connections are established in Theorems~\ref{thm:szego mapping type II} and~\ref{thm:Geronimus} in Section~\ref{ss:SzegoMap}.

To achieve this goal, we first develop the underlying framework, introducing the relevant notions and establishing several structural results that will be used throughout the paper. In Section~\ref{ss:SzegoNEW}, we show that the polynomials $\Phi_{\bm{n};\bm{m}}$ satisfy a family of nearest-neighbour recurrence relations, which can be viewed as a natural extension of the classical Szeg\H{o} recurrences for orthogonal polynomials on the unit circle. In Section~\ref{ss:CC} we derive a system of nearest-neighbour compatibility relations satisfied by the recurrence coefficients and the orthogonal polynomials. Section~\ref{ss:extra} is devoted to a collection of identities that follow from these relations; although somewhat technical in nature, they provide the main tools required for the proof of the Geronimus relations. In Section~\ref{ss:Heine} we establish analogues of the Heine determinantal formulas for the Laurent polynomials $\Phi_{\bm{n};\bm{m}}$, as well as for all associated recurrence coefficients. Section~\ref{ss:CD} contains the Christoffel--Darboux formula.
The results of Sections~\ref{ss:SzegoNEW}, \ref{ss:CC}, and~\ref{ss:CD} follow the same pattern as for the multiple orthogonal polynomials $\Phi_{\bm{n}}$, with the main distinction that $\Phi_{\bm{n};\bm{m}}$ involves two independently varying multi-indices. 
The corresponding proofs follow closely arguments developed in our~\cite{KVMOPUC}, with appropriate modifications to accommodate the present more general framework. 

We also note that in this work we adopt the broader setting, where orthogonality with respect to measures is replaced by orthogonality with respect to moment linear functionals. This generalization does not introduce any additional technical complications, and all statements and proofs extend naturally to it.  Readers who prefer the classical setting may simply interpret each functional as integration with respect to a measure. 


Laurent orthogonal polynomials have also been considered in several very recent works~\cites{KVMLOPUC1,KNik,HueMan}. 
The polynomials in~\cite{HueMan} correspond to the mixed multiple orthogonality setting. While their work studies recurrence relations, Christoffel--Darboux formula, and spectral transformations, only for very specific multi-indices do their polynomials coincide with ours, because their focus is restricted to step-line indices. Furthermore, because they do not introduce the II${}^*$ and I${}^*$ families, their formulas take a substantially different form.
The polynomials $\varphi_{2\bm{n}}$ in~\cite{KVMLOPUC1,KNik} coincide with the Laurent polynomials $\Phi_{\bm{n};\bm{n}}$ studied here, and normality (existence and uniqueness) for $\varphi_{\bm{n}}$ was established for broad classes of measures, including Angelesco, AT, and Nikishin systems on the unit circle. Notice that this establishes existence and uniqueness of $\Phi_{\bm{n};\bm{n}}$ and of the Hermite--Pad\'{e} approximants for ~\eqref{eq:intro generalized two point hermite pade eq 1}--\eqref{eq:intro generalized two point hermite pade eq 2} in the case $\bm{n}=\bm{m}$. The extent to which these results generalize to arbitrary indices $(\bm{n};\bm{m})$ is currently an interesting open problem. Another natural question concerns the location of zeros, which is closely related to the issue of normality, as indicated by the ideas in~\cite{KVInterlacing}. We leave these questions for future work.

\subsubsection*{Acknowledgements}
R.K. is grateful to  Maxim Derevyagin, Brian Simanek, and Sergey Denisov for useful discussions regarding Szeg\H{o}'s mapping.

\section{Orthogonality Relations and Normality}\label{ss:normality}

Let $r$ sequences $(c_{k,j})_{k\in\Z}$, $j=1,\ldots,r$, of arbitrary complex numbers be given; we refer to them as moments. We define the corresponding moment functionals $L_j$ as the linear maps on the space of Laurent polynomials with complex coefficients, uniquely determined by
\begin{equation}\label{eq:moments linear functional}
         L_j[w^{-k}] = c_{k,j} , \qquad k\in\Z, \qquad j = 1,\dots,r. 
    \end{equation}

Of special importance is the case when a functional $L_j$ is induced by a probability measure on the unit circle $\T$
\begin{equation}\label{eq:functionals given by measures}
        L_j[w^{-k}] = \int_{\T} w^{-k} d\mu_j(w), \qquad 
        k \in \Z. 
\end{equation}


In what follows let $\Z$ be the set of integers, $\N$ be the set of non-negative integers. Given two vectors $\bm{u},\bm{v}\in \Z^r$, we write $\bm{u}\ge \bm{v}$ if $u_j\ge v_j$ for all $j=1,\ldots,r$.
Finally, we denote $|\bm{v}| = v_1+\ldots + v_r$. Note that we do {\it not} take the absolute value of $v_j$, even though $v_j$ may be negative; this is done on purpose. In particular, $\bm{u}\ge \bm{v}$ implies $|\bm{u}|\ge |\bm{v}|$.

Given two $\Z^r$ indices $\bm{n}$, $\bm{m}$, let us denote 
\begin{equation}
     \mathfrak{C}_{2r}
     =
     \left\{
     (\bm{n};\bm{m}) \in\Z^r\times\Z^r : 
     n_j+m_j\ge 0 \mbox{ for all } j
     \right\}
\end{equation}
for the remainder of the text.
    
For $(\bm{n};\bm{m})\in\mathfrak{C}_{2r}$, define the $(|\bm{n}|+|\bm{m}|)\times (|\bm{n}|+|\bm{m}|)$ matrix
\begin{equation}\label{eq:T}
    T_{\bm{n};\bm{m}} =\left( \begin{array}{ccc}
        c_{\abs{\bm{m}}-m_1,1} & \cdots & c_{-\abs{\bm{n}}-m_1+1,1} \\
        c_{\abs{\bm{m}}-m_1+1,1} & \cdots & c_{-\abs{\bm{n}}-m_1+2,1} \\
        \vdots & \ddots & \vdots \\
        c_{\abs{\bm{m}}+n_1-1,1} & \cdots & c_{-\abs{\bm{n}}+n_1,1} 
        \\[0.2em]
        \hdashline[0.5pt/2pt]
        \\[-1.2em]
        & \vdots \\[0.2em]
        \hdashline[0.5pt/2pt]
        \\[-1.0em]
        c_{\abs{\bm{m}}-m_r,r} & \cdots & c_{-\abs{\bm{n}}-m_r+1,r} \\
        c_{\abs{\bm{m}}-m_r+1,r} & \cdots & c_{-\abs{\bm{n}}-m_r+2,r} \\
        \vdots & \ddots & \vdots \\
        c_{\abs{\bm{m}}+n_r-1,r} & \cdots & c_{-\abs{\bm{n}}+n_r,r}
    \end{array}\right). 
\end{equation}
If $n_j=-m_j$ for some $j$, then the corresponding $j$-th block in the above matrix is empty. Furthermore, if $\bm{n}=-\bm{m}$, then we formally take $T_{\bm{n};-\bm{n}} = 1$.

We say that $(\bm{n};\bm{m})\in \mathfrak{C}_{2r}$ 
is normal if $\det T_{\bm{n};\bm{m}} \ne 0$. It should not be surprising that this notion is related to the uniqueness of appropriately normalized Laurent orthogonal polynomials, as described next.

\begin{prop}\label{prop:normality}
Let $\bm{L} = (L_1,\dots,L_r)$ be a system of linear functionals and $(\bm{n};\bm{m})\in \mathfrak{C}_{2r}$, with $\bm{n}\ne -\bm{m}$. 
\begin{itemize}
    \item[{(i)}] $\det T_{\bm{n};\bm{m}} \ne 0$ if and only if there exists a unique Laurent polynomial $\Phi_{\bm{n};\bm{m}}$ of the form
    \begin{equation}\label{eq:moprlIIspan}
    \Phi_{\bm{n};\bm{m}}(z) = z^{\abs{\bm{n}}} + \ldots + \alpha_{\bm{n};\bm{m}}z^{-\abs{\bm{m}}},
    \end{equation}  
    that satisfies 
    \begin{equation}\label{eq:moprlII}
    L_j[\Phi_{\bm{n};\bm{m}}(w)w^{-k}] = 0, 
     \quad -m_j \le k \le  n_j-1, \quad 1\le j \le r.
    \end{equation}
\item[{(ii)}] $\det T_{\bm{n};\bm{m}} \ne 0$ if and only if there exists a unique Laurent polynomial $\Phi^*_{\bm{n};\bm{m}}$ of the form
    \begin{equation}\label{eq:moprlII*span}
    \Phi_{\bm{n};\bm{m}}^*(z) = \beta_{\bm{n};\bm{m}}z^{\abs{\bm{n}}} + \ldots + z^{-\abs{\bm{m}}}
\end{equation} 
    that satisfies 
    \begin{equation}\label{eq:moprlII*}
        L_j[\Phi_{\bm{n};\bm{m}}^*(w)w^{-k}] = 0, 
        \quad -m_j+1\le k \le n_j, \quad 1\le j \le r.
\end{equation}
     \item[{(iii)}] $\det T_{\bm{n};\bm{m}} \ne 0$ if and only if there is a unique vector $\bm{\Xi}_{\bm{n};\bm{m}} = (\Xi_{\bm{n};\bm{m},1},\dots,\Xi_{\bm{n};\bm{m},r})$ of Laurent polynomials of the form
    \begin{equation}\label{eq:moprlIspan}
    \Xi_{\bm{n};\bm{m},j}(z) \in \operatorname{span}\big\{z^k\big\}_{k = -n_j}^{m_j-1},
    \end{equation}
    that satisfies 
    \begin{subnumcases}{\sum_{j = 1}^r L_j[\Xi_{\bm{n};\bm{m},j}(w)w^{-k}] =}  
            \label{eq:moprlI}
            0, & 
            $-\abs{\bm{n}}+1\le k \le \abs{\bm{m}}-1$,
            \\
            \label{eq:moprlInorm}
            1, & $k  =-\abs{\bm{n}}$.
    \end{subnumcases}
    \item[{(iv)}] $\det T_{\bm{n};\bm{m}} \ne 0$ if and only if there is a unique vector $\bm{\Xi}^*_{\bm{n};\bm{m}} = (\Xi^*_{\bm{n};\bm{m},1},\dots,\Xi^*_{\bm{n};\bm{m},r})$ of Laurent polynomials of the form
    \begin{equation}\label{eq:moprlI*span}
    \Xi^*_{\bm{n};\bm{m},j}(z) \in \operatorname{span}\big\{z^k\big\}_{k = -n_j+1}^{m_j},
    \end{equation}
    that satisfies 
    \begin{subnumcases}{\sum_{j = 1}^r L_j[\Xi^*_{\bm{n};\bm{m},j}(w)w^{-k}] =}         
            \label{eq:moprlI*}
            0, & 
            $-\abs{\bm{n}}+1\le k \le \abs{\bm{m}}-1$,
            \\
            \label{eq:moprlI*norm}
            1, & $k  =\abs{\bm{m}}$.
    \end{subnumcases}
\end{itemize}
\end{prop}
\begin{proof}
    Treating the coefficients of $\Phi_{\bm{n},\bm{m}}$ in~\eqref{eq:moprlIIspan} as unknowns, write the system of equations \eqref{eq:moprlII}. Its $(\abs{\bm{n}}+\abs{\bm{m}}) \times (\abs{\bm{n}}+\abs{\bm{m}})$ coefficient matrix is $T_{\bm{n};\bm{m}}$ which proves (i). 

    Similar arguments work in other cases, leading to the same coefficient matrix $T_{\bm{n};\bm{m}}$ in (ii) and to the transpose of $T_{\bm{n};\bm{m}}$ in (iii)  and (iv).
\end{proof}
\begin{rem}
    Observe that in (iii) and (iv) we are using $\bm{n}$-indices with the negative sign and $\bm{m}$-indices with the positive sign, as opposed to the cases  (i) and (ii). This is to ensure duality, i.e., that uniqueness of $\Phi_{\bm{n};\bm{m}}$  is equivalent to the uniqueness of $\bm{\Xi}_{\bm{n};\bm{m}}$, as the proposition above shows. 
\end{rem}
\begin{rem}\label{rem:r=1} 
In the case of orthogonality with respect to  one probability measure   (so that $r=1$, $\bm{n} = n \in \Z$, $\bm{m} = m \in \Z$, $n+m\ge0$), it is easy to see that 
\begin{alignat*}{2}
    \Phi_{n;m}(z) &= z^{-m}\Phi_{n+m}(z), \qquad  &\Phi_{n;m}^*(z) &= z^{-m}\Phi^*_{n+m}(z),
    \\
    {\Xi}_{n;m}(z) & = \tfrac{1}{\kappa_{n+m-1}} z^{-n}\Phi^*_{n+m-1}(z),
    \qquad 
    &{\Xi}^*_{n;m}(z) &= \tfrac{1}{\kappa_{n+m-1}} z^{-n+1} \Phi_{n+m-1}(z),
\end{alignat*}
where $\kappa_j = L[\Phi_j(w)w^{-j}] = L[|\Phi_j(w)|^2]$. 
\end{rem}

The Laurent polynomials described in   (i), (ii),  (iii), and (iv) will be called type II, type II$^*$, I, and I$^*$, respectively. For the case of indices of the form $(\bm{n};-\bm{n})$ (for any $\bm{n}\in\Z^r$), which are all normal, it is natural to take $\Phi_{\bm{n};-\bm{n}}(z) = \Phi^*_{\bm{n};-\bm{n}}(z) = 1$, $\bm{\Xi}_{\bm{n};-\bm{n}}(z) = \bm{\Xi}^*_{\bm{n};-\bm{n}}(z) = \bm{0}$. 

If $(\bm{n};\bm{m})$ is normal, we will write $\alpha_{\bm{n};\bm{m}}$ for the $z^{-\abs{\bm{m}}}$-coefficient of $\Phi_{\bm{n};\bm{m}}(z)$ and $\beta_{\bm{n};\bm{m}}$ for the $z^{\abs{\bm{n}}}$-coefficient of $\Phi_{\bm{n};\bm{m}}^*(z)$. These should be viewed as the generalized Verblunsky 
recurrence coefficients, see Section~\ref{ss:SzegoNEW}.




Let us introduce the following notation. Given a function $f(z)$ of a complex variable, let
\begin{equation}\label{eq:sharp1}
        f^\sharp(z) = \overline{f(1/\bar{z})}.
\end{equation}
Similarly, given a moment functional $L$ with $L[w^{-k}] = c_k$, let
\begin{equation}\label{eq:sharp2}
        L^\sharp[w^{-k}] = \overline{L[w^k]} = \bar{c}_{-k}, \qquad k\in\Z.
\end{equation}
Equivalently,
\begin{equation}\label{eq:sharp3}
    L^\sharp[f(z)] = \overline{L[f^\sharp(z)]}.
\end{equation}

\begin{prop}\label{thm:reversal vs star polynomials}
    $(\bm{n};\bm{m})\in \mathfrak{C}_{2r}$ 
    is normal with respect to a system $\bm{L}=(L_1,\ldots,L_r)$ if and only if $(\bm{m};\bm{n})$ is normal with respect to $\bm{L}^\sharp =(L_1^\sharp,\ldots,L_r^\sharp)$. In this case,
\begin{align}
\label{eq:symm1}
    &\Phi_{\bm{n};\bm{m}}(z,\bm{L}^\sharp)  = \Phi^*_{\bm{m};\bm{n}}(z,\bm{L}) ^\sharp,
    \\
    \label{eq:symm2}
    &\Phi^*_{\bm{n};\bm{m}}(z,\bm{L}^\sharp) = \Phi_{\bm{m};\bm{n}}(z,\bm{L})^\sharp,
    \\
    \label{eq:symm3}
    &\bm\Xi_{\bm{n};\bm{m}}(z,\bm{L}^\sharp)  = \bm\Xi^*_{\bm{m};\bm{n}}(z,\bm{L}) ^\sharp,
    \\
    \label{eq:symm4}
    &\bm\Xi^*_{\bm{n};\bm{m}}(z,\bm{L}^\sharp) = \bm\Xi_{\bm{m};\bm{n}}(z,\bm{L})^\sharp,
    \\
    \label{eq:symm5}
    &\alpha_{\bm{n};\bm{m}}(\bm{L}^\sharp) = \bar{\beta}_{\bm{m};\bm{n}}(\bm{L}),
    \\
    \label{eq:symm6}
    &\beta_{\bm{n};\bm{m}}(\bm{L}^\sharp) = \bar{\alpha}_{\bm{m};\bm{n}}(\bm{L}).
    \end{align}

    Furthermore, if each $L_j$ satisfies $c_{k,j} = \bar{c}_{-k,j}$ for all $k\in\N$ (e.g., if it is induced by a probability measure \eqref{eq:functionals given by measures} on the unit circle $\T$), then $(\bm{n};\bm{m})$ is normal if and only if $(\bm{m};\bm{n})$ is normal, and 
    \begin{equation}\label{eq:symm7}
        \Phi_{\bm{n};\bm{m}}^\sharp = \Phi_{\bm{m};\bm{n}}^*,
        \qquad
        \bm{\Xi}_{\bm{n};\bm{m}}^\sharp = \bm{\Xi}_{\bm{m};\bm{n}}^*, 
        \qquad 
        \alpha_{\bm{n};\bm{m}} = \bar{\beta}_{\bm{m};\bm{n}}
    \end{equation}
    hold.
\end{prop}
\begin{proof}
    It is easy to see from~\eqref{eq:T} that $\det T_{\bm{n};\bm{m}}(\bm{L}) = \pm \overline{\det T_{\bm{m};\bm{n}}(\bm{L}^\sharp)}$, which implies the normality claim.  
    Equalities~\eqref{eq:symm1}--\eqref{eq:symm4} follow from the definitions and~\eqref{eq:sharp3}.
    Equalities~\eqref{eq:symm1}--\eqref{eq:symm2} then imply~\eqref{eq:symm5}--\eqref{eq:symm6}. 

    In the special case when the moments $c_{k,j}$ in~\eqref{eq:moments linear functional} satisfy $c_{k,j} = \bar{c}_{-k,j}$ (in particular this holds if $L_j$ corresponds to a positive measure on $\T$), then $L_j = L_j^\sharp$. Combining this with~\eqref{eq:symm1},~\eqref{eq:symm3}, and~\eqref{eq:symm5} gives \eqref{eq:symm7}.
\end{proof}
\begin{rem}\label{rem:annoying1}
    The type II and I polynomials  $\Phi_{\bm{n}}$, $\Phi_{\bm{n}}^*$, $\bm{\Lambda}_{\bm{n}}$, $\bm{\Lambda}_{\bm{n}}^*$ from~\cite{KVMOPUC}  are  related to our Laurent polynomials via
\begin{alignat}{3}
    \label{eq:mopucII}
    &\Phi_{\bm{n}}(z) = \Phi_{\bm{n};\bm{0}}(z),\qquad &&  \Phi_{\bm{n}}^*(z) = \Phi^*_{\bm{n};\bm{0}}(z), \\ 
    \label{eq:mopucI}
    &\bm{\Lambda}_{\bm{n}}(z) = z^{-1} \bm{\Xi}^*_{\bm{0};\bm{n}}(z), \qquad &&  \bm{\Lambda}_{\bm{n}}^*(z) = \bm{\Xi}_{\bm{0};\bm{n}}(z).
\end{alignat}
    The same polynomials $\Phi_{\bm{n}}$, $\bm{\Lambda}_{\bm{n}}$  were also studied earlier in~\cite{MOPUC1,MOPUC2}.
\end{rem}
\begin{rem}\label{rem:annoying2}
    While the relations~\eqref{eq:mopucII} are natural, those in~\eqref{eq:mopucI} may at first appear somewhat unexpected. This can be explained by two factors. First, the polynomials $\bm\Lambda_{\bm{n}}$ analyzed in~\cite{MOPUC1,MOPUC2,KVMOPUC} for systems of positive measures are not the most natural objects in the setting of general linear functionals $L_j$ since their uniqueness  (normality) is {\it not} equivalent to that of $\Phi_{\bm{n}}$, but instead to that of $\Phi_{\bm{n}}$ corresponding to the functionals $L_j^\sharp$ (see~\eqref{eq:sharp2}).
    This explains the reversed order of indices $(\bm{0};\bm{n})$ in ~\eqref{eq:mopucI}. Second, the additional factor $z^{-1}$ is, in fact, natural, as it leads to a more aesthetically consistent formulation of the Hermite--Pad\'{e} problem~\eqref{eq:HPtypeI1swapped}--\eqref{eq:HPtypeI2swapped}, the type I recurrence relations~\eqref{eq:SzegoI1}--\eqref{eq:SzegoI2} (compare with~\eqref{eq:SzegoII1}--\eqref{eq:SzegoII2}), and the Christoffel--Darboux formula~\eqref{eq:cd formula1}--\eqref{eq:cd formula2} (compare with~\cite[Thm 6.1]{KVMOPUC}). 
\end{rem}

\section{A General Two-Point Hermite--Padé Problem of Type II}\label{ss:general type II hermite pade}


In this section we discuss the Hermite--Pad\'{e} approximation problem associated with the Laurent polynomials defined in the previous section.  

To each linear functional $L_j$, $j=1,\dots,r$, given in~\eqref{eq:moments linear functional}, we assign the formal power series
\begin{equation}
\label{eq:FGeneral}
    F_j^{(0)}(z) = c_{0,j} + 2\sum_{k = 1}^{\infty}c_{k,j}z^k, \quad F_j^{(\infty)}(z) = -c_{0,j} - 2\sum_{k = 1}^{\infty}c_{-k,j}z^{-k}.
\end{equation}
Since
\begin{equation}\label{eq:kernel}
        \frac{w+z}{w-z}
        =
        \begin{cases}
            1+ 2\sum_{k=1}^\infty \frac{z^k}{w^k},\qquad  & \mbox{if } |z|<|w|, \\
            -1 - 2\sum_{k=1}^\infty \frac{w^k}{z^k}, \qquad & \mbox{if } |z|>|w|, \\
        \end{cases}
\end{equation} 
the pair $(F_j^{(0)},F_j^{(\infty)})$ can formally be viewed as the generalized Carath\'{e}odory function
\begin{equation}\label{eq:generalized caratheodory function}
    F_j(z) = L_j[(w+z)(w-z)^{-1}].
\end{equation} 
In the case when $L_j$ is associated with a probability measure $\mu_j$ as in~\eqref{eq:functionals given by measures}, then~\eqref{eq:FGeneral} are the power series expansions of the Carath\'{e}odory function $F_j(z)$ \eqref{eq:F} of $\mu_j$ at $0$ and at $\infty$. 

Let two multi-indices $\bm{n} = (n_1,\dots,n_r)\in\N^r$, $\bm{m} = (m_1,\dots,m_r)\in\N^r$ 
be given
. Consider the generalized two-point Hermite--Padé problem of finding the Laurent polynomial
\begin{equation}\label{eq:spanFull}
    \Phi_{\bm{n};\bm{m}}
    \in \operatorname{span}\big\{z^k\big\}_{k = -\abs{\bm{m}}}^{\abs{\bm{n}}},
\end{equation}
that simultaneously satisfies
\begin{alignat}{3}
\label{eq:generalized two point hermite pade eq 1}
    & \Phi_{\bm{n};\bm{m}}(z)F_j^{(0)}(z) + \Psi_{\bm{n};\bm{m},j}(z) = \mathcal{O}(z^{n_j}),  \qquad &&z \rightarrow 0, \\\label{eq:generalized two point hermite pade eq 2}
    & \Phi_{\bm{n};\bm{m}}(z)F_j^{(\infty)}(z) + \Psi_{\bm{n};\bm{m},j}(z) = \mathcal{O}(z^{-m_j-1}), \qquad &&z \rightarrow \infty,
\end{alignat}
for some Laurent polynomials $\Psi_{\bm{n};\bm{m},1},\ldots,\Psi_{\bm{n};\bm{m},r} \in \operatorname{span}\big\{z^k\big\}_{k = -\abs{\bm{m}}}^{\abs{\bm{n}}}$.

Note that here we take $\bm{n}\ge\bm{0}$ and $\bm{m}\ge \bm{0}$, which is more restrictive than in the previous section. This is related to the fact that $F^{(0)}$ and $F^{(\infty)}$ start from the $z^0$-term, and therefore the problem~\eqref{eq:spanFull}, \eqref{eq:generalized two point hermite pade eq 1}, \eqref{eq:generalized two point hermite pade eq 2} can become ill-posed if $n_j<0$ or $m_j<0$.

\begin{thm}\label{thm:generalized hermite pade orthogonality}
    Any Laurent polynomial $\Phi_{\bm{n};\bm{m}}$ that solves ~\eqref{eq:spanFull}, \eqref{eq:generalized two point hermite pade eq 1}, \eqref{eq:generalized two point hermite pade eq 2} 
    automatically satisfies the orthogonality conditions~\eqref{eq:moprlII},
    and $\Psi_{\bm{n};\bm{m},j}$ is given by
    \begin{equation}\label{eq:psiGeneral}
        \Psi_{\bm{n};\bm{m},j}(z) = L_j\Big[\frac{w+z}{w-z}\big(\Phi_{\bm{n};\bm{m}}(w) - \Phi_{\bm{n};\bm{m}}(z)\big)\Big] + L_j[\Phi_{\bm{n};\bm{m}}(w)].
   \end{equation}
   Conversely, any Laurent polynomial $\Phi_{\bm{n},\bm{m}} \in \operatorname{span}\set{z^k}_{k = -\abs{\bm{m}}}^{\abs{\bm{n}}}$ that satisfies \eqref{eq:moprlII} solves the Hermite--Padé problem ~\eqref{eq:spanFull}, \eqref{eq:generalized two point hermite pade eq 1}, \eqref{eq:generalized two point hermite pade eq 2}  together with the Laurent polynomial \eqref{eq:psiGeneral}.
\end{thm}
\begin{rem}
    Combining this theorem with the results in~\cite{KVMLOPUC1,KNik}, we can see that up to a normalization constant, the solution to the Hermite--Pad\'{e} problem ~\eqref{eq:spanFull}, \eqref{eq:generalized two point hermite pade eq 1}, \eqref{eq:generalized two point hermite pade eq 2}  with $\bm{m}=\bm{n}$ is unique if $\bm{L}$ is an Angelesco, AT, or Nikishin system (with $r=2$) on the unit circle.
\end{rem}
\begin{proof}
    Let us write $\Phi_{\bm{n};\bm{m}}$ and $\Psi_{\bm{n};\bm{m},j}$ in the form
    \begin{align}\label{eq:coefficients of phi and psi}
        & \Phi_{\bm{n};\bm{m}}(z) =
        \kappa_{-\abs{\bm{m}}}z^{-\abs{\bm{m}}}+\dots + \kappa_{\abs{\bm{n}}}z^{\abs{\bm{n}}} , 
        \\ 
        \label{eq:coefficients of y}
        & \Psi_{\bm{n};\bm{m},j}(z) =
        \lambda_{-\abs{\bm{m}},j}z^{-\abs{\bm{m}}} + \dots +\lambda_{\abs{\bm{n}},j}z^{\abs{\bm{n}}}.
    \end{align}
    where there are finitely many non-zero $\lambda_{k,j}$ coefficients, so that~\eqref{eq:coefficients of y} is in fact a Laurent polynomial.
    Then we get
    \begin{align*}
        & \Phi_{\bm{n};\bm{m}}(z)F_j^{(0)}(z) = a_{-\abs{\bm{m}},j}z^{-\abs{\bm{m}}} + \dots + a_{n_j-1,j}z^{n_j-1} + \mathcal{O}(z^{n_j}), \quad z \rightarrow 0, \\
        & \Phi_{\bm{n};\bm{m}}(z)F_j^{(\infty)}(z) = b_{\abs{\bm{n}},j}z^{\abs{\bm{n}}} + \dots + b_{-m_j,j}z^{-m_j} + \mathcal{O}(z^{-m_j-1}), \quad z \rightarrow \infty,
    \end{align*}
    where $a_{-\abs{\bm{m}},j},\dots,a_{n_j-1,j}$ and $b_{-m_j,j},\dots,b_{\abs{\bm{n}}-1,j}$ are given by
    \begin{align}
        \label{eq:coefficients of series type II}
         & a_{k,j} = \kappa_kc_{0,j} + 2\kappa_{k-1}c_{1,j} + \dots + 2\kappa_{-\abs{\bm{m}}} c_{\abs{\bm{m}}+k,j},  
         \\ 
         \label{eq:coefficients of series 2}
         & b_{k,j} = -\kappa_kc_{0,j} - 2\kappa_{k+1} c_{-1,j} - \dots - 2 \kappa_{\abs{\bm{n}}} c_{-\abs{\bm{n}}+k,j}.
    \end{align}
    For \eqref{eq:generalized two point hermite pade eq 1}--\eqref{eq:generalized two point hermite pade eq 2} to hold we must have
    \begin{align}\label{eq:coefficients type II eq 1}
        \lambda_{k,j} = a_{k,j}, \qquad &k = -\abs{\bm{m}},\dots,n_j-1, \\\label{eq:coefficients type II eq 2}
        \lambda_{k,j} = b_{k,j}, \qquad &k = -m_j,\dots,\abs{\bm{n}}.
    \end{align}
    This implies 
    \begin{equation}
        a_{k,j} = b_{k,j}, \qquad k = -m_j,\dots,n_j-1,
    \end{equation}
    and by \eqref{eq:moments linear functional} and \eqref{eq:coefficients of series type II}--\eqref{eq:coefficients of series 2}, this turns into \eqref{eq:moprlII}. 
    
    From \eqref{eq:coefficients type II eq 1}--\eqref{eq:coefficients type II eq 2} we also see that $\Psi_{\bm{n};\bm{m},j}$ is uniquely determined by $\Phi_{\bm{n};\bm{m}}$. To compute $\Psi_{\bm{n};\bm{m},j}$, write $\widetilde{\Psi}_{\bm{n};\bm{m},j}$ for the right hand side of \eqref{eq:psiGeneral}, and define $R_{\bm{n};\bm{m},j}$ by
    \begin{equation}
        R_{\bm{n};\bm{m},j}(z) = L_j\Big[\frac{w+z}{w-z}\Phi_{\bm{n};\bm{m}}(w) \Big].
    \end{equation}
    Again, this should be treated as the two formal power series
    \begin{align}\label{eq:secondKindFunctionTypeII 1}
    R^{(0)}_{\bm{n};\bm{m},j}(z) 
        & =  L_j[\Phi_{\bm{n};\bm{m}}(w)]+ 2\sum_{k=1}^\infty z^k L_j[w^{-k}\Phi_{\bm{n};\bm{m}}(w)],  & z & \to 0,  
        \\ \label{eq:secondKindFunctionTypeII 2}
    R^{(\infty)}_{\bm{n};\bm{m},j}(z) & =  -L_j[\Phi_{\bm{n};\bm{m}}(w)] - 2\sum_{k=1}^\infty z^{-k} L_j[w^{k}\Phi_{\bm{n};\bm{m}}(w)], & z & \to \infty. 
    \end{align}
    By \eqref{eq:moprlII}
    we get
    \begin{align}\label{eq:R asymptotics type II eq 1}
        R^{(0)}_{\bm{n};\bm{m},j}(z) 
        & =  \mathcal{O}(z^{n_j}),  & z & \to 0, \\ \label{eq:R asymptotics type II eq 2}
        R^{(\infty)}_{\bm{n};\bm{m},j}(z) & = -L_j[\Phi_{\bm{n};\bm{m}}(w)] + \mathcal{O}(z^{-m_j-1}),  & z & \to \infty.
    \end{align}
    The constant term in~\eqref{eq:R asymptotics type II eq 2} is only relevant if $n_j=0$ and otherwise can be removed by orthogonality~\eqref{eq:moprlII}. 
    
    Note that $\frac{w+z}{w-z}(\Phi_{\bm{n};\bm{m}}(w) - \Phi_{\bm{n};\bm{m}}(z))$ is a Laurent polynomial in $z$ (as well as in $w$). By~\eqref{eq:kernel},
    this polynomial coincides with the $z$-power series at $0$
    \begin{equation}
        \Big(1+ 2\sum_{k=1}^\infty \frac{z^k}{w^k} \Big) \Phi_{\bm{n};\bm{m}}(w) -  \Big(1+ 2\sum_{k=1}^\infty \frac{z^k}{w^k} \Big)  \Phi_{\bm{n};\bm{m}}(z),
    \end{equation}
    as well as with the $z$-power series at $\infty$
    \begin{equation}
       \Big(-1- 2\sum_{k=1}^\infty \frac{w^k}{z^k} \Big) \Phi_{\bm{n};\bm{m}}(w) - \Big(-1- 2\sum_{k=1}^\infty \frac{w^k}{z^k} \Big)  \Phi_{\bm{n};\bm{m}}(z). 
    \end{equation}
    Therefore we can apply the $L_j$ functional to obtain
    \begin{align}\label{eq:psiVsRGeneral}
        \widetilde{\Psi}_{\bm{n};\bm{m},j}(z) & = {R}^{(0)}_{\bm{n};\bm{m},j}(z)  -  \Phi_{\bm{n};\bm{m}}(z) F^{(0)}_j(z)  + L_j[\Phi_{\bm{n},\bm{m}}(w)], 
        \\
        \widetilde{\Psi}_{\bm{n};\bm{m},j}(z) & = {R}^{(\infty)}_{\bm{n};\bm{m},j}(z)  -  \Phi_{\bm{n};\bm{m}}(z) F^{(\infty)}_j(z) + L_j[\Phi_{\bm{n},\bm{m}}(w)].
    \end{align}
    By \eqref{eq:R asymptotics type II eq 1}-\eqref{eq:R asymptotics type II eq 2} we get that $\widetilde{\Psi}_{\bm{n};\bm{m},j}$ satisfies equalities~\eqref{eq:generalized two point hermite pade eq 1}--\eqref{eq:generalized two point hermite pade eq 2} which implies
    $\widetilde{\Psi}_{\bm{n};\bm{m},j} = \Psi_{\bm{n};\bm{m},j}$. 

    Conversely, given $\Phi_{\bm{n};\bm{m}}$ and the corresponding $\Psi_{\bm{n};\bm{m},j}$ as in~\eqref{eq:psiGeneral}, define $R_{\bm{n};\bm{m},j}$ as in~\eqref{eq:secondKindFunctionTypeII 1}--\eqref{eq:secondKindFunctionTypeII 2}. As above, we show that it satisfies~\eqref{eq:R asymptotics type II eq 1} and ~\eqref{eq:R asymptotics type II eq 2}, which then becomes ~\eqref{eq:generalized two point hermite pade eq 1}, \eqref{eq:generalized two point hermite pade eq 2}.
\end{proof}

Let us  also introduce the related Hermite--Padé problem
\begin{alignat}{3}
\label{eq:generalized two point hermite pade star eq 1}
    & \Phi_{\bm{n};\bm{m}}^*(z)F_j^{(0)}(z) -\Psi_{\bm{n};\bm{m},j}^*(z) = \mathcal{O}(z^{n_j+1}), \qquad &&z \rightarrow 0, \\\label{eq:generalized two point hermite pade star eq 2}
    & \Phi_{\bm{n};\bm{m}}^*(z)F_j^{(\infty)}(z) - \Psi_{\bm{n};\bm{m},j}^*(z) = \mathcal{O}(z^{-m_j}), \qquad &&z \rightarrow \infty,
\end{alignat}
for Laurent polynomials 
\begin{equation}\label{eq:spanFull*}
 \Phi_{\bm{n};\bm{m}}^*, \Psi_{\bm{n};\bm{m},1}^*, \ldots, \Psi_{\bm{n};\bm{m},r}^* \in \operatorname{span}\big\{z^k\big\}_{k = -\abs{\bm{m}}}^{\abs{\bm{n}}}.
\end{equation}
\begin{thm}
    Any Laurent polynomial $\Phi_{\bm{n};\bm{m}}^*$ that solves~\eqref{eq:generalized two point hermite pade star eq 1}, \eqref{eq:generalized two point hermite pade star eq 2}, \eqref{eq:spanFull*} must satisfy the orthogonality relations \eqref{eq:moprlII*}, and $\Psi_{\bm{n};\bm{m},j}^*$ is given by
    \begin{equation}\label{eq:psiGeneral*}
        \Psi_{\bm{n};\bm{m},j}^*(z) = L_j\Big[\frac{w+z}{w-z}\big(\Phi_{\bm{n};\bm{m}}^*(z) - \Phi_{\bm{n};\bm{m}}^*(w) \big)\Big] + L_j[\Phi_{\bm{n};\bm{m}}^*(w)].
    \end{equation}
     Conversely, any Laurent polynomial $\Phi^*_{\bm{n},\bm{m}} \in \operatorname{span}\set{z^k}_{k = -\abs{\bm{m}}}^{\abs{\bm{n}}}$ that satisfies \eqref{eq:moprlII*} solves the Hermite--Padé problem ~\eqref{eq:generalized two point hermite pade star eq 1}, \eqref{eq:generalized two point hermite pade star eq 2}, \eqref{eq:spanFull*}   together with the Laurent polynomial \eqref{eq:psiGeneral*}.
\end{thm}

\begin{proof}
   Use the same methods as in the proof of Theorem~\ref{thm:generalized hermite pade orthogonality}.
    The minus sign in~\eqref{eq:psiGeneral*} arises from the fact that, when expanding~\eqref{eq:secondKindFunctionTypeII 1} and~\eqref{eq:secondKindFunctionTypeII 2} for $\Phi_{\bm{n};\bm{m}}^*$, the additional term $L_j[\Phi_{\bm{n};\bm{m}}^*(w)]$ vanishes for $R^{(\infty)}$ but remains for $R^{(0)}$.
\end{proof}


\section{A General Two-Point Hermite-Padé Problem of Type I}\label{ss:HPI}

Let again $\bm{n}\in\N^r$ and $\bm{m}\in\N^r$ be given. 
We now consider the two-point Hermite--Padé problem which can be viewed as dual to the one in the previous section~\eqref{eq:spanFull}, \eqref{eq:generalized two point hermite pade eq 1}, \eqref{eq:generalized two point hermite pade eq 2}. 

Let $F_j^{(0)}$, $F_j^{(\infty)}$ be the formal series~\eqref{eq:FGeneral} associated with linear functionals $L_j$ in~\eqref{eq:moments linear functional}. As before, if $L_j$ corresponds to integration with respect to a probability measure as in~\eqref{eq:functionals given by measures}, then $F_j^{(0)}$ and $F_j^{(\infty)}$ are the power series expansions of the Carath\'{e}odory function $F_j(z)$ \eqref{eq:F} of $\mu_j$ at $0$ and at $\infty$.
We are interested in the solutions to the following approximation problem:
\begin{alignat}{3}
\label{eq:HPtypeI1swapped}
    & \sum_{j = 1}^r \Xi_{\bm{n};\bm{m},j}(z)F^{(0)}_j(z) + \Upsilon_{\bm{n};\bm{m}}(z) = \mathcal{O}(z^{\abs{\bm{m}}}), \qquad && z \rightarrow 0, \\
\label{eq:HPtypeI2swapped}
    & \sum_{j = 1}^r \Xi_{\bm{n};\bm{m},j}(z)F_j^{(\infty)}(z) + \Upsilon_{\bm{n};\bm{m}}(z) = \mathcal{O}(z^{-\abs{\bm{n}}}), \qquad && z \rightarrow \infty,
\end{alignat}
for Laurent polynomials 
\begin{equation}
    \label{eq:HPtypeIspanSwapped}
    \Xi_{\bm{n};\bm{m},j}(z) \in \operatorname{span}\big\{z^{k}\big\}_{k=-n_j}^{m_j-1},
    \qquad
     \Upsilon_{\bm{n};\bm{m}}(z) \in \operatorname{span}\big\{z^{k}\big\}_{k=-N}^{M-1},
\end{equation}
where $N:=\max\{n_1,\ldots,n_r\}$, $M:=\max\{m_1,\ldots,m_r\}$.


\begin{thm}
    If the Laurent polynomials 
    $\bm{\Xi}_{\bm{n};\bm{m}} = (\Xi_{\bm{n};\bm{m},1},\dots,\Xi_{\bm{n};\bm{m},r})$ and $\Upsilon_{\bm{n};\bm{m}}$ solve~\eqref{eq:HPtypeI1swapped}, \eqref{eq:HPtypeI2swapped}, \eqref{eq:HPtypeIspanSwapped},  then $\bm{\Xi}_{\bm{n};\bm{m}}$ satisfies the orthogonality relations \eqref{eq:moprlI}, and $\Upsilon_{\bm{n};\bm{m}}$ is given by
    \begin{equation}\label{eq:upsilon}
        \Upsilon_{\bm{n};\bm{m}}(z) =  \sum_{j = 1}^r L_j\Big[\frac{w+z}{w-z}\big(\Xi_{\bm{n};\bm{m},j}(w)-\Xi_{\bm{n};\bm{m},j}(z)\big)\Big] + {\delta_{\bm{n};\bm{m}} \sum_{j = 1}^r L_j[\Xi_{\bm{n};\bm{m},j}(w)]},
    \end{equation}
   {where the constant $\delta_{\bm{n};\bm{m}}$ is given by
   \begin{equation}\label{eq:delta}
       \delta_{\bm{n};\bm{m}} = \begin{cases}
           1, \qquad & \bm{m}=\bm{0} \mbox{ and } \bm{n}\neq\bm{0}, \\
           -1, \qquad & \bm{n}=\bm{0} \mbox{ and } \bm{m}\neq\bm{0}, \\
           0, \qquad &  \mbox{otherwise}.
       \end{cases}
   \end{equation}}

    Conversely, any vector $\bm{\Xi}_{\bm{n};\bm{m}} = (\Xi_{\bm{n};\bm{m},1},\dots,\Xi_{\bm{n};\bm{m},r})$ with $\Xi_{\bm{n};\bm{m},j} \in \operatorname{span}\set{z^k}_{k = -n_j}^{m_j-1}$, that satisfies 
    ~\eqref{eq:moprlI}, solves the Hermite--Pad\'{e} problem~\eqref{eq:HPtypeI1swapped}, \eqref{eq:HPtypeI2swapped}, \eqref{eq:HPtypeIspanSwapped} together with the Laurent polynomial~\eqref{eq:upsilon}.
\end{thm}


\begin{proof}
    Write $\Xi_{\bm{n};\bm{m},j}$ and $\Upsilon_{\bm{n};\bm{m}}$ in the form 
    \begin{align*}\label{eq:lamdba}
         \Xi_{\bm{n};\bm{m},j} (z)&=  \kappa_{-n_j,j}z^{-n_j} + \kappa_{-n_j+1,j}z^{-n_j+1} + \ldots + \kappa_{m_j-1,j}z^{m_j-1}, \\ 
        \Upsilon_{\bm{n};\bm{m}}(z) & =\lambda_{-N}z^{-N} + \lambda_{-N+1}z^{-N+1} + \ldots + \lambda_{M-1}z^{M-1} .
    \end{align*}
    We put $\kappa_{k,j} = 0$ if $k \ge m_j$ or $k\le -n_j-1$, and 
    $\lambda_{k} = 0$ if $k \ge M$ or $k \le -N -1$. We then have 
    \begin{alignat*}{3}
        & \Xi_{\bm{n};\bm{m},j} F_j^{(0)} = a_{-n_j,j}z^{-n_j} + \ldots + a_{|\bm{m}|-1,j}z^{|\bm{m}|-1} + \mathcal{O}(z^{|\bm{m}|}), \quad &&z \rightarrow 0, \\
        & \Xi_{\bm{n};\bm{m},j} F_j^{(\infty)} = b_{m_j-1,j}z^{m_j-1} + \ldots + b_{-|\bm{n}|+1,j}z^{-|\bm{n}|+1} + \mathcal{O}(z^{-|\bm{n}|}), \quad  &&z \rightarrow \infty,
    \end{alignat*}
    where we have
    \begin{equation}
        \label{eq:coefficients of series 1 I}
        a_{k,j}  = \kappa_{k,j} c_{0,j}+ 2\kappa_{k-1,j}c_{1,j} + \ldots + 2\kappa_{-n_j,j} c_{n_j+k,j}
    \end{equation}
    for $k = -n_j,\ldots,|\bm{m}|-1$, and $a_{k,j}=0$ for $k <-n_j$, and
    \begin{equation}
        \label{eq:coefficients of series 2 I}
        b_{k,j} = -\kappa_{k,j} c_{0,j} - 2\kappa_{k+1,j} c_{-1,j} - \ldots - 2 \kappa_{m_j-1,j} c_{-m_j+k+1,j}
    \end{equation}
    for $k = -|\bm{n}|+1,\ldots,m_j-1$, and $b_{k,j} =0$ for $ k>m_j-1$.
    
    

    To get \eqref{eq:HPtypeI1swapped} and \eqref{eq:HPtypeI2swapped} we then necessarily need
    \begin{alignat}{3}
        \label{eq:lambda1}
        \lambda_k + \sum_{j=1}^r a_{k,j} & = 0, \qquad && k\le |\bm{m}|-1,
        \\
        \label{eq:lambda2}
        \lambda_k + \sum_{j=1}^r b_{k,j} & = 0, \qquad && k\ge |\bm{n}|-1.
    \end{alignat}
    In particular, this implies
    \begin{equation}
        \label{eq:PadeToOrthoEq I}
          \sum_{j=1}^r a_{k,j} = \sum_{j=1}^r b_{k,j},  \quad k = -|\bm{n}|+1,\dots,|\bm{m}|-1,
    \end{equation}
    which, combined with~\eqref{eq:coefficients of series 1 I},  \eqref{eq:coefficients of series 2 I} and~\eqref{eq:moments linear functional}, turns into~\eqref{eq:moprlI}. From \eqref{eq:lambda1} and \eqref{eq:lambda2}, we also see that $\Upsilon_{\bm{n};\bm{m}}$ is uniquely determined by $\bm{\Xi}_{\bm{n};\bm{m}}$. Observe that \eqref{eq:lambda1} and \eqref{eq:lambda2} is consistent with $\lambda_{k} = 0$ if $k \ge M$ (since in this case $b_{k,j}=0$ for all $j$) or $k \le -N -1$ (in this case $a_{k,j}=0$ for all $j$).
    

    Now consider the formal power series
    \begin{align}
    \label{eq:secondKindFunctionTypeI 1}
    R^{(0)}_{\bm{n};\bm{m}}(z) 
    & = \sum_{j=1}^r \Big( L_j[\Xi_{\bm{n};\bm{m},j}(w)]+ 2\sum_{k=1}^\infty z^k L_j[\Xi_{\bm{n};\bm{m},j}(w) w^{-k}] \Big) , 
    \\
    \label{eq:secondKindFunctionTypeI 2}
     R^{(\infty)}_{\bm{n};\bm{m}}(z)   & =     - \sum_{j=1}^r \Big( L_j[\Xi_{\bm{n};\bm{m},j}(w)] +
            2\sum_{k=1}^\infty z^{-k} L_j[\Xi_{\bm{n};\bm{m},j}(w) w^{k}] \Big). 
    \end{align}

    Using \eqref{eq:moprlI*}, we get 
    \begin{alignat}{3}\label{eq:R0 I}
        R_{\bm{n};\bm{m}}^{(0)}(z) & = \sum_{j=1}^r  L_j[\Xi_{\bm{n};\bm{m},j}(w)] + 
        \mathcal{O}(z^{|\bm{m}|}), \qquad && z \to 0,
        \\
        \label{eq:Rinfty I}
        R_{\bm{n};\bm{m}}^{(\infty)}(z)
        & = 
         - \sum_{j=1}^r  L_j[\Xi_{\bm{n};\bm{m},j}(w)] +
         \mathcal{O}(z^{-|\bm{n}|}), \qquad && z \to \infty.
     \end{alignat}
    The constant terms in~\eqref{eq:R0 I} are only relevant if $\bm{n}=\bm{0}$ and otherwise can be removed. Similarly, the constant terms
    in~\eqref{eq:Rinfty I} are only relevant if $\bm{m}=\bm{0}$. 

    Now denote $\widetilde{\Upsilon}_{\bm{n};\bm{m}}(z)$ to be the right-hand side of~\eqref{eq:upsilon}. It is clear that it is a Laurent polynomial in $\operatorname{span}\big\{z^k\big\}_{k = -N}^{M-1}$. Expanding algebraically the right-hand side of~\eqref{eq:upsilon} into power series, we get  
    \begin{align}
        \widetilde{\Upsilon}_{\bm{n};\bm{m}}(z)
        & =
        R^{(0)}_{\bm{n};\bm{m}}(z) - \sum_{j=1}^r \Xi_{\bm{n};\bm{m},j}(z) F_j^{(0)}(z) +\delta_{\bm{n};\bm{m}} \sum_{j = 1}^r L_j[\Xi_{\bm{n};\bm{m},j}(w)],
        \\
        \widetilde{\Upsilon}_{\bm{n};\bm{m}}(z)
        & =
        R^{(\infty)}_{\bm{n};\bm{m}}(z) - \sum_{j=1}^r \Xi_{\bm{n};\bm{m},j}(z) F^{(\infty)}_j(z) +\delta_{\bm{n};\bm{m}} \sum_{j = 1}^r L_j[\Xi_{\bm{n};\bm{m},j}(w)].
    \end{align}
    
    Combining this with~\eqref{eq:R0 I} and~\eqref{eq:Rinfty I}, we see that
\eqref{eq:HPtypeI1swapped}--\eqref{eq:HPtypeI2swapped} hold with the same $\bm{\Xi}_{\bm{n};\bm{m}}$ but with $\Upsilon_{\bm{n};\bm{m}}$ replaced by $\widetilde\Upsilon_{\bm{n};\bm{m}}$.  
    But the proof showed that this uniquely determines all the coefficients of $\widetilde\Upsilon_{\bm{n};\bm{m}}$. This shows that $\widetilde\Upsilon_{\bm{n};\bm{m}} = \Upsilon_{\bm{n};\bm{m}}$ and proves~\eqref{eq:upsilon}.
    
    Conversely, given $\bm{\Xi}_{\bm{n};\bm{m}}$ and the corresponding $\Upsilon_{\bm{n};\bm{m}}$ as in~\eqref{eq:upsilon}, define $R_{\bm{n};\bm{m}}$ as in~\eqref{eq:secondKindFunctionTypeI 1}--\eqref{eq:secondKindFunctionTypeI 2}. As above, we show that it satisfies~\eqref{eq:R0 I} and ~\eqref{eq:Rinfty I}, which then becomes ~\eqref{eq:HPtypeI1swapped}, \eqref{eq:HPtypeI2swapped}.
\end{proof}

The corresponding type I problem for the ``starred'' polynomials looks as follows. We are looking for the Laurent polynomials 
\begin{equation}
    \label{eq:HPtypeI*spanSwapped}
     \Xi^*_{\bm{n};\bm{m},j}(z) \in \operatorname{span}\big\{z^{k}\big\}_{k=-n_j+1}^{m_j},
    \qquad \Upsilon^*_{\bm{n};\bm{m}}(z) \in \operatorname{span}\big\{z^{k}\big\}_{k=-N+1}^{M},
\end{equation}
that satisfy
\begin{alignat}{3}
    \label{eq:HPtypeI*1swapped}
    & \sum_{j = 1}^r \Xi^*_{\bm{n};\bm{m},j}(z)F_j^{(0)}(z) - \Upsilon^*_{\bm{n};\bm{m}}(z) = \mathcal{O}(z^{\abs{\bm{m}}}), \qquad &&z \rightarrow 0, \\
    \label{eq:HPtypeI*2swapped}
    &\sum_{j = 1}^r \Xi^*_{\bm{n};\bm{m},j}(z)F_j^{(\infty)}(z) -\Upsilon^*_{\bm{n};\bm{m}}(z) = \mathcal{O}(z^{-\abs{\bm{n}}}),  \qquad &&z \rightarrow \infty.
\end{alignat}

\begin{thm}
    Any vector $\bm{\Xi}_{\bm{n};\bm{m}}^* = (\Xi_{\bm{n};\bm{m},1}^*,\dots,\Xi_{\bm{n};\bm{m},r}^*)$ that solves ~\eqref{eq:HPtypeI*spanSwapped}, \eqref{eq:HPtypeI*1swapped}, \eqref{eq:HPtypeI*2swapped} must satisfy the orthogonality relations \eqref{eq:moprlI*}, and $\Upsilon_{\bm{n};\bm{m}}^*$ is given by
    \begin{equation}\label{eq:upsilon star}
        \Upsilon_{\bm{n};\bm{m}}^*(z) = \sum_{j = 1}^r \Big( L_j\Big[\frac{w+z}{w-z}\big( \Xi_{\bm{n};\bm{m},j}^*(z) - \Xi_{\bm{n};\bm{m},j}^*(w) \big)\Big] + \delta_{\bm{n};\bm{m}} L_j[\Xi_{\bm{n};\bm{m},j}^*(w)] \Big),
    \end{equation}
    {where $\delta_{\bm{n};\bm{m}}$ is the same as in \eqref{eq:upsilon}.}

    Conversely, any vector $\bm{\Xi}_{\bm{n};\bm{m}}^*$ with $\Xi^*_{\bm{n};\bm{m},j} \in \operatorname{span}\big\{z^{k}\big\}_{k=-n_j+1}^{m_j}$ that satisfies  \eqref{eq:moprlI*} solves  ~\eqref{eq:HPtypeI*spanSwapped}, \eqref{eq:HPtypeI*1swapped}, \eqref{eq:HPtypeI*2swapped} together with~\eqref{eq:upsilon star}.
\end{thm}


    
  



\section{Szeg\H{o} Recurrence Relations}\label{ss:SzegoNEW}
Let us again allow $(\bm{n};\bm{m})\in\mathfrak{C}_{2r}$ as in Section~\ref{ss:normality}.
Here we establish that the Laurent polynomials $\Phi_{\bm{n};\bm{m}}$ satisfy a family of recurrence relations, which follow directly from our earlier results in~\cite{KVMOPUC}. That work was restricted to orthogonality with respect to positive measures as in~\eqref{eq:functionals given by measures} and to {\it non-Laurent} polynomials. The underlying arguments, however, extend with only minor modifications to the general case of moment functionals~\eqref{eq:moments linear functional}. This yields the following generalization.

\begin{thm}\label{thm:Szego1}
    Assuming all the $\mathfrak{C}_{2r}$-indices that appear in the equations below are normal, we have the Szeg\H{o} recurrence relations 
    \begin{align}
         \label{eq:SzegoII1}
        \vphantom{\sum_{j = 1}^{r}} 
        & \Phi_{\bm{n};\bm{m}}^*(z) = \Phi_{\bm{n}-\bm{e}_k;\bm{m}}^*(z) + \beta_{\bm{n};\bm{m}}z\Phi_{\bm{n}-\bm{e}_k;\bm{m}}(z), 
        \\ \label{eq:SzegoII2}
        & \Phi_{\bm{n};\bm{m}}(z) = \alpha_{\bm{n};\bm{m}}\Phi_{\bm{n};\bm{m}}^*(z) + \sum_{j = 1}^r \rho_{\bm{n};\bm{m},j}z\Phi_{\bm{n}-\bm{e}_j;\bm{m}}(z), 
    \end{align}
    and
    \begin{align}
        \label{eq:SzegoI1}
        & \vphantom{\sum_{j = 1}^{r}} 
        \bm{\Xi}^*_{\bm{n};\bm{m}}(z) = \bm{\Xi}^*_{\bm{n}+\bm{e}_k;\bm{m}}(z) -\alpha_{\bm{n};\bm{m}}z\bm{\Xi}_{\bm{n}+\bm{e}_k;\bm{m}}(z),
        \\
        \label{eq:SzegoI2}
        & \bm{{\Xi}}_{\bm{n};\bm{m}}(z) = -\beta_{\bm{n};\bm{m}}\bm{\Xi}^*_{\bm{n};\bm{m}}(z) + \sum_{j = 1}^{r}\rho_{\bm{n};\bm{m},j}z\bm{\Xi}_{\bm{n}+\bm{e}_j;\bm{m}}(z),
\end{align}
for some complex numbers $\rho_{\bm{n};\bm{m},j}$.
\end{thm}

If we keep the first index constant, then we obtain another family of relations.

\begin{thm}\label{thm:Szego2}
Assuming the indices that appear in the equations below are normal, we have the Szeg\H{o} recurrence relations 
\begin{align}
         \label{eq:SzegoII4}
        \vphantom{\sum_{j = 1}^{r}} 
        & \Phi_{\bm{n};\bm{m}}(z) = \Phi_{\bm{n},\bm{m}-\bm{e}_k}(z) + \alpha_{\bm{n};\bm{m}}z^{-1}\Phi^*_{\bm{n},\bm{m}-\bm{e}_k}(z), 
        \\ \label{eq:SzegoII5}
        & \Phi^*_{\bm{n};\bm{m}}(z) = \beta_{\bm{n};\bm{m}}\Phi_{\bm{n};\bm{m}}(z) + \sum_{j = 1}^r \sigma_{\bm{n};\bm{m},j}z^{-1} \Phi_{\bm{n},\bm{m}-\bm{e}_j}^*(z), 
    \end{align}
    and
    \begin{align}
        \label{eq:SzegoI4}
        & \vphantom{\sum_{j = 1}^{r}} 
        \bm{\Xi}_{\bm{n};\bm{m}}(z) = \bm{\Xi}_{\bm{n};\bm{m}+\bm{e}_k}(z) -\beta_{\bm{n};\bm{m}}z^{-1} \bm{\Xi}^*_{\bm{n};\bm{m}+\bm{e}_k}(z),
        \\
        \label{eq:SzegoI5}
        & \bm{{\Xi}}^*_{\bm{n};\bm{m}}(z) = -\alpha_{\bm{n};\bm{m}}\bm{\Xi}_{\bm{n};\bm{m}}(z) + \sum_{j = 1}^{r}\sigma_{\bm{n};\bm{m},j}z^{-1} \bm{\Xi}^*_{\bm{n};\bm{m}+\bm{e}_j}(z),
\end{align}
for some complex numbers $\sigma_{\bm{n};\bm{m},j}$.
\end{thm}
\begin{rem}
     One can formally put $\Phi_{\bm{n}-\bm{e}_j;\bm{m}} = 0$ and $\rho_{\bm{n};\bm{m},j} = 0$ when $(\bm{n}-\bm{e}_j;\bm{m})\notin \mathfrak{C}_{2r}$. With such a choice, equality~\eqref{eq:SzegoII2} 
     holds even if some of the indices $(\bm{n}-\bm{e}_j;\bm{m})$ 
     fall outside of $\mathfrak{C}_{2r}$. Similar modifications apply to \eqref{eq:SzegoII5}. 
\end{rem}
\begin{proof}[Proof of Theorems~\ref{thm:Szego1} and~\ref{thm:Szego2}]
        The proof of recurrences~\eqref{eq:SzegoII1}and \eqref{eq:SzegoII2}
        goes through the exact same arguments as~\cite[Thm 3.1 and Cor 3.3]{KVMOPUC}. 

        Relations ~\eqref{eq:SzegoII4} and \eqref{eq:SzegoII5}
        follow then from Theorem~\ref{thm:reversal vs star polynomials}. 
        
        The proof of recurrences~\eqref{eq:SzegoI1},and \eqref{eq:SzegoI2}
        follow similar lines to the arguments in~\cite[Thm 4.3]{KVMOPUC} with modifications related to the discussion in Remarks~\ref{rem:annoying1} and \ref{rem:annoying2}. Proposition~\ref{prop:general hermite pade measures type I} below summarizes the main modifications that are needed in order to adapt the proofs from ~\cite[Sect. 4]{KVMOPUC}.
\end{proof}

\begin{prop}\label{prop:general hermite pade measures type I}Assume the indices that appear in the respective relations below are normal. 
    \noindent\hfill
    \begin{itemize}
        \item[(i)] Let $\kappa_{\bm{n};\bm{m},j}$ be the $z^{-n_j}$ coefficient of $\Xi_{\bm{n};\bm{m},j}(z)$ and let $\ell_{\bm{n};\bm{m},j}$ be the $z^{m_j}$ coefficient of $\Xi^*_{\bm{n};\bm{m},j}(z)$. Then
        \begin{align}
            \label{eq:kappa}
            \kappa_{\bm{n}+\bm{e}_j;\bm{m},j} & = \frac{1}{L_j[\Phi_{\bm{n};\bm{m}}(w)w^{-n_j}]},
            \\
            \label{eq:l}
            \ell_{\bm{n}+\bm{e}_j;\bm{m},j} &= \frac{1}{L_j[\Phi^*_{\bm{n};\bm{m}}(w)w^{m_j}]}.
        \end{align}
        

        \item[(ii)] We have the relations
        \begin{alignat}{1}
        \label{eq:rho}
            & \rho_{\bm{n};\bm{m},j}= \frac{L_j[\Phi_{\bm{n};\bm{m}}(w)w^{-n_j}]}{L_j[\Phi_{\bm{n}-\bm{e}_j;\bm{m}}(w)w^{-n_j+1}]},
            \\
            \label{eq:sigma}
            & \sigma_{\bm{n};\bm{m},j}= \frac{L_j[\Phi^*_{\bm{n};\bm{m}}(w)w^{m_j}]}{L_j[\Phi^*_{\bm{n};\bm{m}-\bm{e}_j}(w)w^{m_j-1}]}.
        \end{alignat}
        Furthermore, $\rho_{\bm{n};\bm{m},j} \neq 0$ if and only if $(\bm{n}+\bm{e}_j;\bm{m})$ is normal. Similarly, $\sigma_{\bm{n};\bm{m},j} \neq 0$ if and only if $(\bm{n};\bm{m}+\bm{e}_j)$ is normal.
        
        \item[(iii)] The following relations hold:
        \begin{alignat}{3}
            \label{eq:biorthogonality 1}
            & \sum_{j = 1}^r L_j \big[ \Xi_{\bm{n};\bm{m}{,j}}(w)w^{-\abs{\bm{m}}} \big] = -{\beta}_{\bm{n};\bm{m}}, 
            \\ \label{eq:biorthogonality 2}
            & \sum_{j = 1}^r L_j \big[ \Xi_{\bm{n};\bm{m}{,j}}^*(w) w^{\abs{\bm{n}}} \big] = -{\alpha}_{\bm{n};\bm{m}}.
        \end{alignat}
    \end{itemize}
\end{prop}        
\begin{rem}
    Note that the right hand side of \eqref{eq:rho} only requires normality of $(\bm{n};\bm{m})$ and $(\bm{n}-\bm{e}_j;\bm{m})$. It is possible to derive \eqref{eq:SzegoII2} without appealing to normality of indices of the form $(\bm{n}-\bm{e}_{\ell};\bm{m})$, $\ell\ne j$. In that case any choice of polynomials $\Phi_{\bm{n}-\bm{e}_{\ell};\bm{m}}$,  $\ell\ne j$, will work for the relation to hold. When $(\bm{n}-\bm{e}_j;\bm{m})$ is normal for a given $j$ then $\rho_{\bm{n};\bm{m},j}$ is uniquely defined and can be computed to \eqref{eq:rho}. Similarly, $\sigma_{\bm{n};\bm{m}}$ is defined as long as $(\bm{n};\bm{m})$ and $(\bm{n};\bm{m}-\bm{e}_j)$ are normal. 
\end{rem}

\begin{rem}\label{rem:reversal}
    It is easy to see that in the setting of 
         Theorem~\ref{thm:reversal vs star polynomials}, we have
         \begin{equation}
             \rho_{\bm{n};\bm{m},j}(\bm{L}^\sharp) = \bar\sigma_{\bm{m};\bm{n},j}(\bm{L}), \qquad \sigma_{\bm{n};\bm{m},j}(\bm{L}^\sharp) = \bar\rho_{\bm{m};\bm{n},j}(\bm{L}),
        \end{equation}
        and if each $L_j$ is associated with a probability measure then
        \begin{equation}
             \rho_{\bm{n};\bm{m},j}= \bar\sigma_{\bm{m};\bm{n},j} .
        \end{equation}
\end{rem}

\section{Compatibility relations}\label{ss:CC}

\begin{prop}\label{pr:compPoly1}
Assuming the indices that appear in the equations below are normal, and $k \neq \ell$, then there is a complex number $\gamma_{\bm{n};\bm{m}}^{k\ell}$ such that
\begin{align}\label{eq:CCII}
    & \Phi_{\bm{n}+\bm{e}_k;\bm{m}}-\Phi_{\bm{n}+\bm{e}_\ell;\bm{m}} = \gamma_{\bm{n};\bm{m}}^{k\ell}\Phi_{\bm{n};\bm{m}},
\\
\label{eq:CCI}
    & \bm\Xi_{\bm{n}-\bm{e}_k;\bm{m}}-\bm\Xi_{\bm{n}-\bm{e}_\ell;\bm{m}} = \gamma_{\bm{n}-\bm{e}_k-\bm{e}_\ell;\bm{m}}^{k\ell}\bm\Xi_{\bm{n};\bm{m}}.
\end{align}

Similarly, for some complex number $\eta_{\bm{n};\bm{m}}^{k\ell}$,
\begin{align}\label{eq:CCII2}
    & \Phi^*_{\bm{n};\bm{m}+\bm{e}_k}-\Phi^*_{\bm{n};\bm{m}+\bm{e}_\ell} = \eta_{\bm{n};\bm{m}}^{k\ell}\Phi^*_{\bm{n};\bm{m}},
    \\
    \label{eq:CCIInew}
    & \bm\Xi^*_{\bm{n};\bm{m}-\bm{e}_k}-\bm\Xi^*_{\bm{n};\bm{m}-\bm{e}_\ell} = \eta_{\bm{n};\bm{m}-\bm{e}_k-\bm{e}_\ell}^{k\ell}\bm\Xi^*_{\bm{n};\bm{m}}.
\end{align}
\end{prop}

\begin{proof}
    Check orthogonality relations and degree restrictions of $\Phi_{\bm{n}+\bm{e}_k;\bm{m}}-\Phi_{\bm{n}+\bm{e}_\ell;\bm{m}}$, and compare with $\Phi_{\bm{n};\bm{m}}$. The first relation then follows by normality of $(\bm{n};\bm{m})$, and same arguments prove the other relations. 
\end{proof}

\begin{rem}
    Similarly to the previous section, we have, for $k\ne \ell$,
    \begin{align}\label{eq:gamma}
    \gamma^{k\ell}_{\bm{n};\bm{m}} &= \frac{L_\ell[\Phi_{\bm{n}+\bm{e}_k;\bm{m}}(w)w^{-n_\ell}]}{L_\ell[\Phi_{\bm{n};\bm{m}}(w)w^{-n_\ell}]},
    \\\label{eq:eta}
    \eta^{k\ell}_{\bm{n};\bm{m}} &= \frac{L_\ell[\Phi^*_{\bm{n};\bm{m}+\bm{e}_k}(w)w^{m_\ell}]}{L_\ell[\Phi^*_{\bm{n};\bm{m}}(w)w^{m_\ell}]},
    \end{align}
    and  
    \begin{equation}
             \gamma_{\bm{n};\bm{m}}^{k\ell}(\bm{L}^\sharp) = \bar\eta_{\bm{m};\bm{n}}^{k\ell}(\bm{L}), \qquad \eta_{\bm{n};\bm{m}}^{k\ell}(\bm{L}^\sharp) = \bar\gamma_{\bm{m};\bm{n},j}^{k\ell}(\bm{L}).
    \end{equation}
    Furthermore, $\gamma^{k\ell}_{\bm{n};\bm{m}} \neq 0$ if and only if $(\bm{n}+\bm{e}_k+\bm{e}_\ell;\bm{m})$ is normal, and $\eta_{\bm{n};\bm{m}}^{k\ell} \neq 0$ if and only if $(\bm{n};\bm{m}+\bm{e}_k+\bm{e}_\ell)$ is normal. The constants in \eqref{eq:CCI} and~\eqref{eq:CCIInew} exist even if $(\bm{n}-\bm{e}_k-\bm{e}_\ell;\bm{m})$, resp. $(\bm{n};\bm{m}-\bm{e}_k-\bm{e}_\ell)$, is not normal (see \cite{KVMOPUC} for the details).
\end{rem}

\begin{thm}
Assuming normality of all indices required for the recurrence coefficients below to be defined, we have the partial difference equations
     \begin{align} 
    \label{eq:compatibility condition 2}
    & \alpha_{\bm{n};\bm{m}}\beta_{\bm{n};\bm{m}} + \sum_{j = 1}^r \rho_{\bm{n};\bm{m},j} = 1,
    \\
    \label{eq:compatibility condition 1}
    & 
    \vphantom{\sum_{j = 1}^r }
    \alpha_{\bm{n}+\bm{e}_k;\bm{m}} - \alpha_{\bm{n}+\bm{e}_\ell;\bm{m}}
    = \alpha_{\bm{n};\bm{m}} \gamma_{\bm{n};\bm{m}}^{k\ell},
    \\
    \label{eq:another comp}
    \vphantom{\sum_{j = 1}^r }
    &\beta_{\bm{n}+\bm{e}_\ell;\bm{m}} - \beta_{\bm{n}+\bm{e}_k;\bm{m}}
    = \beta_{\bm{n}+\bm{e}_\ell+\bm{e}_k;\bm{m}} \gamma_{\bm{n};\bm{m}}^{k\ell},
    \\
    \vphantom{\sum_{j = 1}^r }
    \label{eq:compatibility condition 3.1}
    & \rho_{\bm{n};\bm{m},k} \gamma_{\bm{n};\bm{m}}^{k\ell}= \rho_{\bm{n}+\bm{e}_\ell;\bm{m},k}, \gamma_{\bm{n}-\bm{e}_k;\bm{m}}^{k\ell} .
    \end{align}

Similarly,
    \begin{align} 
    \label{eq:compatibility condition 2_another}
    & \alpha_{\bm{n};\bm{m}}\beta_{\bm{n};\bm{m}} + \sum_{j = 1}^r \sigma_{\bm{n};\bm{m},j} = 1,
    \\
    \label{eq:compatibility condition 1_another}
    & 
    \vphantom{\sum_{j = 1}^r }
    \alpha_{\bm{n};\bm{m}+\bm{e}_\ell} - \alpha_{\bm{n};\bm{m}+\bm{e}_k}
    = \alpha_{\bm{n};\bm{m}+\bm{e}_\ell+\bm{e}_k} \eta_{\bm{n};\bm{m}}^{k\ell},
    \\
    \vphantom{\sum_{j = 1}^r }
    &\beta_{\bm{n};\bm{m}+\bm{e}_k} - \beta_{\bm{n};\bm{m}+\bm{e}_\ell}
    = \beta_{\bm{n};\bm{m}} \eta_{\bm{n};\bm{m}}^{k\ell},
    \\
    \vphantom{\sum_{j = 1}^r }
    \label{eq:compatibility condition 3.1_another}
    & \sigma_{\bm{n};\bm{m},k} \eta_{\bm{n};\bm{m}}^{k\ell}= \sigma_{\bm{n};\bm{m}+\bm{e}_\ell,k}, \eta_{\bm{n};\bm{m}-\bm{e}_k}^{k\ell} .
    \end{align}
\end{thm}

\begin{proof}
    \eqref{eq:compatibility condition 2} follows by comparing leading coefficients in \eqref{eq:SzegoII2}, and \eqref{eq:compatibility condition 1} follows by comparing the lowest coefficient in \eqref{eq:CCII}. \eqref{eq:another comp} can be obtained from \eqref{eq:CCI} and \eqref{eq:biorthogonality 1}. \eqref{eq:compatibility condition 3.1} is immediate from \eqref{eq:rho} and \eqref{eq:gamma}. Similar arguments then imply the dual relations \eqref{eq:compatibility condition 2_another}-\eqref{eq:compatibility condition 3.1_another}.
\end{proof}


\section{Consequences of the Szeg\H{o} Relations}\label{ss:extra}

Aside from the main recurrence relations of the previous sections, we want to mention some other recurrence relation that may be useful. We use two of them in the proof of the Geronimus relations for the Szeg\H{o} mapping in Section \ref{ss:SzegoMap}.

\begin{prop}
Assuming all the indices appearing in the equations below are normal, we have recurrence relations
    \begin{align}\label{eq:SzegoII7}
        & \Phi_{\bm{n}+\bm{e}_k;\bm{m}}^*(z) = (1 - \alpha_{\bm{n};\bm{m}+\bm{e}_k}\beta_{\bm{n}+\bm{e}_k;\bm{m}})\Phi_{\bm{n};\bm{m}}^*(z) + \beta_{\bm{n}+\bm{e}_k;\bm{m}}z\Phi_{\bm{n};\bm{m}+\bm{e}_k}(z), \\
        & \label{eq:SzegoII8}\Phi_{\bm{n};\bm{m}+\bm{e}_k}(z) = (1 - \alpha_{\bm{n};\bm{m}+\bm{e}_k}\beta_{\bm{n}+\bm{e}_k;\bm{m}})\Phi_{\bm{n};\bm{m}}(z) + \alpha_{\bm{n};\bm{m}+\bm{e}_k}z^{-1}\Phi_{\bm{n}+\bm{e}_k;\bm{m}}^*(z), 
        \end{align}
        and 
        \begin{align}& \bm{\Xi}_{\bm{n}-\bm{e}_k;\bm{m}}^*(z) = (1-\alpha_{\bm{n}-\bm{e}_k;\bm{m}}\beta_{\bm{n};\bm{m}-\bm{e}_k})\bm{\Xi}_{\bm{n};\bm{m}}^*(z) - \alpha_{\bm{n}-\bm{e}_k;\bm{m}}z\bm{\Xi}_{\bm{n};\bm{m}-\bm{e}_k}(z), \\
        & \bm{\Xi}_{\bm{n};\bm{m}-\bm{e}_k}(z) = (1-\alpha_{\bm{n}-\bm{e}_k;\bm{m}}\beta_{\bm{n};\bm{m}-\bm{e}_k})\bm{\Xi}_{\bm{n};\bm{m}}(z) - \beta_{\bm{n};\bm{m}-\bm{e}_k}z^{-1}\bm{\Xi}_{\bm{n}-\bm{e}_k;\bm{m}}^*(z).
    \end{align}
\end{prop}
\begin{proof}
    Multiply \eqref{eq:SzegoII4} for the index $(\bm{n};\bm{m}+\bm{e}_k)$ by $\beta_{\bm{n}+\bm{e}_k;\bm{m}}z$ and substitute the term $\beta_{\bm{n}+\bm{e}_k;\bm{m}}z\Phi_{\bm{n};\bm{m}}(z)$ using \eqref{eq:SzegoII1} to get \eqref{eq:SzegoII7}. The other equations follow similarly. 
\end{proof}

\begin{cor}Assuming normality of the indices that appear below, we have 
\begin{equation}
    \label{eq:one minus alphabeta}
      \frac{L_k[\Phi_{\bm{n},\bm{m}+\bm{e}_k}(w)w^{-n_k}]}{L_k[\Phi_{\bm{n},\bm{m}}(w)w^{-n_k}]}
      =
        \frac{L_k[\Phi_{\bm{n}+\bm{e}_k,\bm{m}}^*(w)w^{m_k}]}{L_k[\Phi_{\bm{n},\bm{m}}^*(w)w^{m_k}]} = 1-\alpha_{\bm{n};\bm{m}+\bm{e}_k}\beta_{\bm{n}+\bm{e}_k;\bm{m}}.
    \end{equation}
    Furthermore, 
    $\alpha_{\bm{n};\bm{m}+\bm{e}_k}\beta_{\bm{n}+\bm{e}_k;\bm{m}} \neq 1$ if and only if $(\bm{n}+\bm{e}_k;\bm{m}+\bm{e}_k)$ is normal.
\end{cor}
\begin{rem}
    If each functional is given by integration with respect to a probability measure on $\T$, then this together with Proposition~\ref{thm:reversal vs star polynomials} shows that $\abs{\alpha_{\bm{n};\bm{n}+\bm{e}_k}} \neq 1$ and $\abs{\beta_{\bm{n}+\bm{e}_k;\bm{n}}} \neq 1$.
\end{rem}

\begin{proof}
    Multiply both sides of \eqref{eq:SzegoII7} by $z^{m_k}$, and multiply \eqref{eq:SzegoII8} by $z^{-n_k}$, then apply $L_k$ to get \eqref{eq:one minus alphabeta}. If these expressions vanish then $\Phi_{\bm{n},\bm{m}+\bm{e}_k}$ satisfies all the orthogonality relations of $(\bm{n}+\bm{e}_k;\bm{m}+\bm{e}_k)$, which contradicts normality. 
\end{proof}

%
%

\begin{prop} Assuming normality of all indices that appear in the respective equations below, 
we have the recurrence relations
    \begin{align}
        \label{eq:SzegoII9}
         z&\Phi_{\bm{n};\bm{m}}(z) = \Phi_{\bm{n}+\bm{e}_k;\bm{m}}(z) - \alpha_{\bm{n}+\bm{e}_k;\bm{m}}\Phi_{\bm{n};\bm{m}}^*(z) + \sum_{j = 1}^r \rho_{\bm{n};\bm{m},j}\gamma_{\bm{n};\bm{m}}^{jk}z\Phi_{\bm{n}-\bm{e}_j;\bm{m}}(z), \\
        \label{eq:SzegoII10}
         z^{-1}&\Phi_{\bm{n};\bm{m}}^*(z) = \Phi_{\bm{n};\bm{m}+\bm{e}_k}^*(z) - \beta_{\bm{n};\bm{m}+\bm{e}_k}\Phi_{\bm{n};\bm{m}}(z) + \sum_{j = 1}^r \sigma_{\bm{n};\bm{m},j}\eta_{\bm{n};\bm{m}}^{jk}z^{-1}\Phi_{\bm{n};\bm{m}-\bm{e}_j}(z),
    \end{align}
    and 
    \begin{align*}
        z &\bm{\Xi}_{\bm{n};\bm{m}}(z) = \bm{\Xi}_{\bm{n}-\bm{e}_k;\bm{m}}(z) + \beta_{\bm{n}-\bm{e}_k;\bm{m}}\bm{\Xi}_{\bm{n};\bm{m}}^*(z) + \sum_{j=1}^r \rho_{\bm{n};\bm{m},j}\gamma_{\bm{n}-\bm{e}_k+\bm{e}_j;\bm{m}}^{jk}z\bm{\Xi}_{\bm{n}+\bm{e}_j;\bm{m}}(z),\\
        z^{-1}&\bm{\Xi}_{\bm{n};\bm{m}}^*(z) = \bm{\Xi}_{\bm{n};\bm{m}-\bm{e}_k}^*(z) + \alpha_{\bm{n};\bm{m}-\bm{e}_k}\bm{\Xi}_{\bm{n};\bm{m}}(z)+\sum_{j=1}^r \sigma_{\bm{n};\bm{m},j}\eta_{\bm{n};\bm{m}-\bm{e}_k+\bm{e}_j}^{jk}z^{-1}\bm{\Xi}_{\bm{n};\bm{m}+\bm{e}_j}(z)
    \end{align*}
\end{prop}

\begin{proof}
    By \eqref{eq:SzegoII2} and \eqref{eq:CCII} we have
    \begin{align*}
        & \Phi_{\bm{n}+\bm{e}_k;\bm{m}}(z) = \alpha_{\bm{n}+\bm{e}_k;\bm{m}}\Phi_{\bm{n}+\bm{e}_k;\bm{m}}^*(z) + \sum_{j = 1}^r \rho_{\bm{n}+\bm{e}_k;\bm{m},j}z\Phi_{\bm{n}+\bm{e}_k-\bm{e}_j;\bm{m}}(z) \\ & 
        = \alpha_{\bm{n}+\bm{e}_k;\bm{m}}\Phi_{\bm{n}+\bm{e}_k;\bm{m}}^*(z) + \sum_{j \neq k}\rho_{\bm{n}+\bm{e}_k;\bm{m},j}z(\Phi_{\bm{n};\bm{m}}(z) + \gamma_{\bm{n}-\bm{e}_j;\bm{m}}^{kj}\Phi_{\bm{n}-\bm{e}_j;\bm{m}}(z)) \\ & \hspace{5cm}+ \rho_{\bm{n}+\bm{e}_k;\bm{m},k}z\Phi_{\bm{n};\bm{m}}(z).
    \end{align*}
    By separating the two terms in the sum and then using \eqref{eq:compatibility condition 2} and \eqref{eq:compatibility condition 3.1} we end up with
    \begin{multline*}
        \Phi_{\bm{n}+\bm{e}_k;\bm{m}}(z) = \alpha_{\bm{n}+\bm{e}_k;\bm{m}}\Phi_{\bm{n}+\bm{e}_k;\bm{m}}^*(z) + (1 - \alpha_{\bm{n}+\bm{e}_k;\bm{m}}\beta_{\bm{n}+\bm{e}_k;\bm{m}})z\Phi_{\bm{n};\bm{m}}(z) \\ - \sum_{j = 1}^r\rho_{\bm{n};\bm{m},j}\gamma_{\bm{n};\bm{m}}^{jk}z\Phi_{\bm{n}-\bm{e}_j;\bm{m}}(z).
    \end{multline*}
    This together with \eqref{eq:SzegoII1} implies the first equality. 
    The second equality then follows by Proposition \ref{thm:reversal vs star polynomials}. The type I formulas can be derived analogously. 
\end{proof}

\begin{rem}
    The recurrence relations can alternatively be derived directly from orthogonality relations, e.g., the first relation follows by checking orthogonality relations and degree restrictions of $\Phi_{\bm{n};\bm{m}}(z) - z^{-1}\Phi_{\bm{n}+\bm{e}_k;\bm{m}}+\alpha_{\bm{n}+\bm{e}_k;\bm{m}}z^{-1}\Phi_{\bm{n};\bm{m}}^*(z)$ (using similar arguments as in the proofs of the Szeg\H{o} recurrence relations of \cite{KVMOPUC}), and then applying \eqref{eq:rho} and \eqref{eq:gamma}. This approach does not need normality of indices of the form $(\bm{n}+\bm{e}_k-\bm{e}_j;\bm{m})$, $j \neq k$. 
\end{rem}

We can eliminate the coefficients $\gamma_{\bm{n};\bm{m}}^{jk}$ in \eqref{eq:SzegoII9} if we multiply both sides by $\alpha_{\bm{n};\bm{m}}$. By applying \eqref{eq:compatibility condition 1} and then using \eqref{eq:SzegoII2} we get
\begin{multline}\label{eq:generalized three term}
    \alpha_{\bm{n};\bm{m}}z\Phi_{\bm{n};\bm{m}}(z) = \alpha_{\bm{n};\bm{m}}\Phi_{\bm{n}+\bm{e}_k;\bm{m}}(z) - \alpha_{\bm{n}+\bm{e}_k;\bm{m}}\Phi_{\bm{n};\bm{m}}(z) \\ + \sum_{j = 1}^r \alpha_{\bm{n}+\bm{e}_j;\bm{m}}\rho_{\bm{n};\bm{m},j}z\Phi_{\bm{n}-\bm{e}_j;\bm{m}}(z).
    \end{multline}
Hence the end result expresses $z\Phi_{\bm{n};\bm{m}}(z)$ in terms of the nearest neighbours, without appealing to any $\Phi^*$-polynomials or $\beta$-coefficients. In the case $\bm{m} = \bm{0}$ this is the main result of \cite{MOPUC2}. In terms of our recurrence coefficients, in the case $r=1$ this turns into exactly the three-term recurrence relation of OPUC \cite[Sect. 1.5]{OPUC1}. In the degenerate case $\alpha_{\bm{n};\bm{m}} = 0$ we get essentially no information from this relation. 

Of course, an analogous formula can be derived for the type I polynomials, as well as reversed versions through Proposition \ref{thm:reversal vs star polynomials}.
    



\section{Heine formulas}\label{ss:Heine}

Assuming $(\bm{n};\bm{m})$ is normal, it is easy to verify the determinantal formula
    \begin{equation}\label{eq:HeineII}
        \Phi_{\bm{n};\bm{m}}(z) = \frac{1}{\det T_{\bm{n};\bm{m}}}
        \det
          \left(\begin{array}{c c c}
        c_{\abs{\bm{m}}-m_1,1} & \cdots & c_{-\abs{\bm{n}}-m_1,1} \\
        c_{\abs{\bm{m}}-m_1+1,1} & \cdots & c_{-\abs{\bm{n}}-m_1+1,1} \\
        \vdots & \ddots & \vdots \\
        c_{\abs{\bm{m}}+n_1-1,1} & \cdots & c_{-\abs{\bm{n}}+n_1-1,1} 
        \\[0.2em]
        \hdashline[0.5pt/2pt]
        \\[-1.2em]
        & \vdots \\[0.2em]
        \hdashline[0.5pt/2pt]
        \\[-1.0em]
        c_{\abs{\bm{m}}-m_r,r} & \cdots & c_{-\abs{\bm{n}}-m_r,r} \\
        c_{\abs{\bm{m}}-m_r+1,r} & \cdots & c_{-\abs{\bm{n}}-m_r+1,r} \\
        \vdots & \ddots & \vdots \\
        c_{\abs{\bm{m}}+n_r-1,r} & \cdots & c_{-\abs{\bm{n}}+n_r-1,r}
        \\[0.2em]
        \hdashline[0.5pt/2pt]
        \\[-1.0em]
        z^{-|\bm{m}|} & \cdots & z^{|\bm{n}|}
    \end{array}\right),
    \end{equation}
    where $T_{\bm{n};\bm{m}}$ is given in~\eqref{eq:T}. Indeed, it is easy to see that the right-hand side of~\eqref{eq:HeineII} is 
    of the form~\eqref{eq:moprlIIspan}. Then multiplying it with $z^{-k}$ for a given $k=-m_j,\ldots,n_j-1$ and applying $L_j$, one obtains a determinant with two identical rows, implying~\eqref{eq:moprlII}.

    Similarly,
    \begin{equation}\label{eq:HeineII*}
        \Phi^*_{\bm{n};\bm{m}}(z) = \frac{(-1)^{|\bm{n}|+|\bm{m}|}}{\det T_{\bm{n};\bm{m}}}
        \det
          \left(\begin{array}{c c c}
        c_{\abs{\bm{m}}-m_1+1,1} & \cdots & c_{-\abs{\bm{n}}-m_1+1,1} \\
        c_{\abs{\bm{m}}-m_1+2,1} & \cdots & c_{-\abs{\bm{n}}-m_1+2,1} \\
        \vdots & \ddots & \vdots \\
        c_{\abs{\bm{m}}+n_1,1} & \cdots & c_{-\abs{\bm{n}}+n_1,1} 
        \\[0.2em]
        \hdashline[0.5pt/2pt]
        \\[-1.2em]
        & \vdots \\[0.2em]
        \hdashline[0.5pt/2pt]
        \\[-1.0em]
        c_{\abs{\bm{m}}-m_r+1,r} & \cdots & c_{-\abs{\bm{n}}-m_r+1,r} \\
        c_{\abs{\bm{m}}-m_r+2,r} & \cdots & c_{-\abs{\bm{n}}-m_r+2,r} \\
        \vdots & \ddots & \vdots \\
        c_{\abs{\bm{m}}+n_r,r} & \cdots & c_{-\abs{\bm{n}}+n_r,r}
        \\[0.2em]
        \hdashline[0.5pt/2pt]
        \\[-1.0em]
        z^{-|\bm{m}|} & \cdots & z^{|\bm{n}|}
    \end{array}\right).
    \end{equation}
    Analogous formulas can be established for the type I polynomials $\bm{\Xi}_{\bm{m};\bm{n}}$ and $\bm{\Xi}^*_{\bm{m};\bm{n}}$ in the same vein as in~\cite{Kui}.

    This leads to the following determinantal expressions for all the recurrence coefficients from Section~\ref{ss:SzegoNEW} and Section \ref{ss:CC}. In what follows, we denote $wL$ to be linear functional (Christoffel transform) defined by $(wL)[w^k]:=L[w^{k+1}]$, $w\bm{L} = (wL_1,\ldots,wL_r)$, and similarly for $w^{-1}L$ and $w^{-1} \bm{L}$. Finally, we denote $T_{\bm{n};\bm{m}}(\bm{M})$ to be the matrix in~\eqref{eq:T} for any system of linear functionals $\bm{M}$.

    \begin{thm}
    \noindent\hfill
        \begin{itemize}
            \item[(i)] If $(\bm{n};\bm{m})\in\mathfrak{C}_{2r}$ is normal, then
            \begin{alignat}{2}
                \alpha_{\bm{n};\bm{m}} &= (-1)^{|\bm{n}|+|\bm{m}|} \frac{\det T_{\bm{n};\bm{m}}(w\bm{L})}{\det T_{\bm{n};\bm{m}}(\bm{L})},
                \\
                \beta_{\bm{n};\bm{m}} &=(-1)^{|\bm{n}|+|\bm{m}|} \frac{\det T_{\bm{n};\bm{m}}(w^{-1}\bm{L})}{\det T_{\bm{n};\bm{m}}(\bm{L})}.
            \end{alignat}

            \item[(ii)] If $(\bm{n};\bm{m}),(\bm{n}-\bm{e}_j;\bm{m})\in\mathfrak{C}_{2r}$ are normal, then
            \begin{equation}
                \rho_{\bm{n};\bm{m},j} =  \frac{\det T_{\bm{n}+\bm{e}_j;\bm{m}} \det T_{\bm{n}-\bm{e}_j;\bm{m}} }{(\det T_{\bm{n};\bm{m}})^2} ,
            \end{equation}
            and if $(\bm{n};\bm{m}),(\bm{n};\bm{m}-\bm{e}_j)\in\mathfrak{C}_{2r}$ are normal, then
            \begin{equation}
                \sigma_{\bm{n};\bm{m},j} = \frac{\det T_{\bm{n};\bm{m}+\bm{e}_j} \det T_{\bm{n};\bm{m}-\bm{e}_j} }{(\det T_{\bm{n};\bm{m}})^2}.
            \end{equation}

            \item[(iii)] If $(\bm{n};\bm{m}),(\bm{n}+\bm{e}_k;\bm{m}),(\bm{n}+\bm{e}_\ell;\bm{m})\in\mathfrak{C}_{2r}$  are normal, and $k < \ell$, then
            \begin{equation}
                \gamma^{k\ell}_{\bm{n};\bm{m}} = \frac{\det T_{\bm{n}+\bm{e}_\ell+\bm{e}_k;\bm{m}} \det T_{\bm{n};\bm{m}} }{\det T_{\bm{n}+\bm{e}_\ell;\bm{m}} \det T_{\bm{n}+\bm{e}_k;\bm{m}}} ,
            \end{equation}
            and if $(\bm{n};\bm{m}),(\bm{n};\bm{m}+\bm{e}_k),(\bm{n};\bm{m}+\bm{e}_\ell)\in\mathfrak{C}_{2r}$  are normal, and $k < \ell$, then
            \begin{equation}
                \eta^{k\ell}_{\bm{n};\bm{m}} = \frac{\det T_{\bm{n};\bm{m}+\bm{e}_\ell+\bm{e}_k} \det T_{\bm{n};\bm{m}} }{\det T_{\bm{n};\bm{m}+\bm{e}_\ell} \det T_{\bm{n};\bm{m}+\bm{e}_k}} .
            \end{equation}
            \item[(iv)] If $(\bm{n};\bm{m}),(\bm{n}+\bm{e}_k;\bm{m}),(\bm{n};\bm{m}+\bm{e}_k)\in\mathfrak{C}_{2r}$  are normal
            \begin{equation}
                1-\alpha_{\bm{n};\bm{m}+\bm{e}_k}\beta_{\bm{n}+\bm{e}_k;\bm{m}} = \frac{\det{T_{\bm{n}+\bm{e}_k;\bm{m}+\bm{e}_k}}\det{T_{\bm{n};\bm{m}}}}{\det{T_{\bm{n}+\bm{e}_k;\bm{m}}}\det{T_{\bm{n};\bm{m}+\bm{e}_k}}}.
            \end{equation}
            
        \end{itemize}
    \end{thm}
\begin{proof}
    For (i) and (ii), we simply need to find the $z^{-|\bm{m}|}$ coefficient of~\eqref{eq:HeineII} and $z^{|\bm{n}|}$ coefficient of~\eqref{eq:HeineII*}.
    
    Now note that from~\eqref{eq:HeineII} and~\eqref{eq:HeineII*} we obtain
    \begin{align}
    \label{eq:norm1}
        L_j[\Phi_{\bm{n};\bm{m}}(w)w^{-n_j}]
        & = (-1)^{\sum_{s=j+1}^r (n_s+m_s)} \frac{\det T_{\bm{n}+\bm{e}_j;\bm{m}}}{\det T_{\bm{n};\bm{m}}},
        \\
        \label{eq:norm2}
        L_j[\Phi^*_{\bm{n};\bm{m}}(w)w^{m_j}]
        & =  (-1)^{\sum_{s=1}^{j-1} (n_s+m_s)} \frac{\det T_{\bm{n};\bm{m}+\bm{e}_j}}{\det T_{\bm{n};\bm{m}}}.
    \end{align}
    Then (ii) follows by combining this with~\eqref{eq:rho} and~\eqref{eq:sigma}, (iii) follows by combining with \eqref{eq:gamma} and \eqref{eq:eta}, and (iv) follows by combining with \eqref{eq:one minus alphabeta}.
\end{proof}


\section{The Christoffel--Darboux formula}\label{ss:CD}

We can establish a generalized Christoffel--Darboux formula by following the same argument as in~\cite[Thm 6.1]{KVMOPUC}. In this formulation, the paths may originate from points other than $\bm{0}$, and the resulting expression takes a more streamlined form due to our choice of definition of $\bm\Xi_{\bm{n};\bm{m}}$, as discussed in Remarks~\ref{rem:annoying1}, \ref{rem:annoying2}.
\begin{thm} 
    Let $\bm{m}\in\N^r$ be fixed, and consider $(\bm{n}_k;\bm{m})_{k = 0}^{N}$ to be a path of $\mathfrak{C}_{2r}$-multi-indices, 
    such that $\bm{n}_0 = \bm{-\bm{m}}$, 
    and $\bm{n}_{k+1}-\bm{n}_k = \bm{e}_{l_k}$ for some $1\le l_k\le r$. 
    Assume all multi-indices on the path are normal, along with all the neighbouring indices that belong to $\mathfrak{C}_{2r}$. Then we have the Christoffel--Darboux formulas
    \begin{equation}\label{eq:cd formula1}
        \begin{aligned}
    (\xi - z)& \sum_{k = 0}^{N-1} \Phi_{\bm{n}_k;\bm{m}}(z)  \bm\Xi_{\bm{n}_{k+1};\bm{m}}(\xi)
    \\
    &= \Phi^*_{\bm{n}_N;\bm{m}}(z) \bm\Xi^*_{\bm{n}_N;\bm{m}}(\xi) - z\xi \sum_{j = 1}^{r}\rho_{\bm{n}_N;\bm{m},j}\Phi_{\bm{n}_N-\bm{e}_j;\bm{m}}(z) \bm\Xi_{\bm{n}_N+\bm{e}_j;\bm{m}}(\xi)
    \end{aligned}
    \end{equation}
    and
    \begin{equation}\label{eq:cd formula2}
    \begin{aligned}
    (z-\xi)& \sum_{k = 0}^{N-1} \Phi^*_{\bm{m};\bm{n}_k}(z)  \bm\Xi^*_{\bm{m};\bm{n}_{k+1}}(\xi) 
    \\
    &= z\xi \Phi_{\bm{m};\bm{n}_N}(z)\bm\Xi_{\bm{m};\bm{n}_N}(\xi) -  \sum_{j = 1}^{r}\sigma_{\bm{m};\bm{n}_N,j}\Phi^*_{\bm{m};\bm{n}_N-\bm{e}_j}(z) \bm\Xi^*_{\bm{m};\bm{n}_N+\bm{e}_j}(\xi).
    \end{aligned}
    \end{equation}
    \begin{proof}
        In the expression
        $$
         \Phi_{\bm{n}_k;\bm{m}}(z)  \bm\Xi_{\bm{n}_{k+1};\bm{m}}(\xi) - z\xi^{-1} \Phi_{\bm{n}_k;\bm{m}}(z)  \bm\Xi_{\bm{n}_{k+1};\bm{m}}(\xi)
        $$
        use the Szeg\H{o} recurrence~\eqref{eq:SzegoII2} followed by~\eqref{eq:SzegoI1} on the first term and~\eqref{eq:SzegoI2} followed by~\eqref{eq:SzegoII1}. Then the arguments from the proof ~\cite[Thm 6.1]{KVMOPUC} go through in a straightforward manner.
        
        The second formula~\eqref{eq:cd formula2} follows from~\eqref{eq:cd formula1} after applying Proposition~\ref{thm:reversal vs star polynomials} and Remark~\ref{rem:reversal}.
    \end{proof}
\end{thm}

\section{Szeg\H{o} Mapping}\label{ss:SzegoMap}

Let us recall the main notions from the theory of orthogonal and multiple orthogonal polynomials on the real line with respect to moment functionals acting on the space of complex polynomials 
(with nonnegative powers, in contrast to the Laurent setting~\eqref{eq:moments linear functional}). To this end, let $M$ be defined by
\begin{equation}\label{eq:moments linear functional real}
         M[x^{k}] = m_{k} , \qquad k\in\N, 
\end{equation}
where $m_{k}$ are arbitrary complex numbers. 

Often $M$ is induced by a probability measure $\gamma$ on the real line $\R$ with all finite moments
\begin{equation}\label{eq:functionals given by measures real}
        M[x^{k}] = \int_{\R} x^{k} d\gamma(x), \qquad 
        k \in \N. 
\end{equation}

Orthogonal polynomials with respect to $M$ are defined by $\deg P_n = n$ and 
$$
M[P_n(x) x^k] = 0, \qquad k=0,1,\ldots,n-1,
$$
Recall~\cite{Chihara} that a moment functional $M$ is called quasi-definite if there is a unique monic $P_n$ and $\deg{P_n} = n$, for each positive integer $n$. This is equivalent to all Hankel matrices $(m_{j+\ell})_{j,\ell=0}^{n-1}$ being invertible. This is automatic if $M$ is associated with a positive measure on $\R$ with finite moments and infinite support, as in~\eqref{eq:functionals given by measures real}.

Similarly, a moment functional $L$ on the space of Laurent polynomials~\eqref{eq:moments linear functional} is called quasi-definite if a monic polynomial $\Phi_n$ with $\deg \Phi_n = n$ and
$$
L[\Phi_n(z) z^{-k}] = 0, \qquad k=0,1,\ldots,n-1,
$$ 
exists and is unique for every $n\in\N$. This is equivalent to all Toeplitz matrices $(c_{j-\ell})_{j,\ell=0}^{n-1}$ being invertible and is automatic if $L$ is associated with a positive measure with infinite support on $\T$ as in~\eqref{eq:functionals given by measures}.

Given a system $\bm{M} = (M_1,\dots,M_r)$ of moment functionals~\eqref{eq:moments linear functional real}, the type II multiple orthogonal polynomials $P_{\bm{n}}$ with respect to the multi-index $\bm{n}$ and the system $\bm{M}$ are non-zero polynomials of at most degree $\abs{\bm{n}}$, satisfying the orthogonality relations
\begin{equation}\label{eq:MOPRLIIreal}
    M_j[P_{\bm{n}}(x)x^k] = 0, \qquad k = 0,\dots,n_j-1, \qquad j = 1,\dots,r.
\end{equation}
Such polynomials always exists, and we say that the index $\bm{n}$ is normal if there exists a unique $P_{\bm{n}}$ with $x^{\abs{\bm{n}}}$-coefficient equal to $1$. 

The type I polynomials are non-zero vectors $\bm{A}_{\bm{n}} = (A_{\bm{n},1},\dots,A_{\bm{n},r})$, except for $\bm{A}_{\bm{0}} = \bm{0}$, where $A_{\bm{n},j}$ are polynomials of degree at most $n_j-1$ for each $j = 1,\dots,r$, and
\begin{equation}\label{eq:MOPRLIreal}
    \sum_{j = 1}^r M_j[A_{\bm{n},j}(x)x^k ] = 
        0, \qquad k = 0,\dots,\abs{\bm{n}}-2.
\end{equation}
$\bm{n} \neq \bm{0}$ is normal if and only if there exists a unique $\bm{A}_{\bm{n}}$ with
\begin{equation}\label{eq:normalization type I moprl}
    \sum_{j = 1}^r M_j[A_{\bm{n},j}(x) x^{\abs{\bm{n}}-1}] = 1. 
\end{equation}
When $\bm{n}$ is normal we always work with the above normalizations for the type I and type II polynomials. 

Assuming all the indices are normal (we then say that the system is perfect), there exist~\cite{NNRR,Ismail} coefficients $a_{\bm{n},j}, b_{\bm{n},j}$, called the nearest neighbour recurrence coefficients, so that 
\begin{equation}\label{eq:one nnr}
xP_{\bm{n}}(x) = P_{\bm{n} + \bm{e}_k}(x) + b_{\bm{n},k}P_{\bm{n}}(x) + \sum_{j = 1}^{r}a_{\bm{n},j}P_{\bm{n} - \bm{e}_j}(x),
\end{equation}
as well as
\begin{equation}\label{eq:type 1 recurrence relation}
x\bm{A}_{\bm{n}}(x) = \bm{A}_{\bm{n} - \bm{e}_k}(x) + b_{\bm{n}-\bm{e}_k,k}\bm{A}_{\bm{n}}(x) + \sum_{j = 1}^{r}a_{\bm{n},j}\bm{A}_{\bm{n} + \bm{e}_j}(x).
\end{equation}


Recall the definition of the Szeg\H{o} mapping in Section~\ref{ss:intro3} for the case of measures. It is easy to extend this construction to linear functionals. Indeed, the equalities
\begin{equation}
    \label{eq:SzegoMapFunctionals}
    L[(w+w^{-1})^k] = M[x^k], \qquad k\in\N
\end{equation}
set up one-to-one correspondence between all linear functionals $M$ on the space of polynomials (i.e.,~\eqref{eq:moments linear functional real}), and all linear functionals $L$ on the space of Laurent polynomials (i.e., ~\eqref{eq:moments linear functional}) with the symmetry
\begin{equation}\label{eq:symmetricFunctionals}
    L[w^k]=L[w^{-k}], \qquad k\in\N.
\end{equation}
We denote this correspondence by $L = \operatorname{Sz}(M)$ and its inverse by $M = \operatorname{Sz}^{-1}(L)$.



Functionals on Laurent polynomials with the symmetry~\eqref{eq:symmetricFunctionals} exhibit a number of properties similar to those that we observed in Proposition~\ref{thm:reversal vs star polynomials} and Remark~\ref{rem:reversal}. 

\begin{prop}\label{rem:Rsymmetry}
    Let $L_1,\ldots,L_r$ be linear functionals on the space of Laurent polynomials, each satisfying the symmetry condition~\eqref{eq:symmetricFunctionals}.
    
    Then $(\bm{n};\bm{m})$ is normal if and only if $(\bm{m};\bm{n})$ is normal, and  
    \begin{alignat}{1}
        & \Phi_{\bm{n};\bm{m}}^*(z)=\Phi_{\bm{m};\bm{n}}(1/z),
        \qquad
        \bm{\Xi}^*_{\bm{n};\bm{m}}(z) = \bm{\Xi}_{\bm{m};\bm{n}}(1/z), 
        \\
        & \label{eq:szego symmetry}
        \alpha_{\bm{n};\bm{m}} = {\beta}_{\bm{m};\bm{n}},
        \qquad
        \rho_{\bm{n};\bm{m},j}= \sigma_{\bm{m};\bm{n},j},
        \qquad
        \gamma^{kl}_{\bm{n};\bm{m},j}= \eta^{kl}_{\bm{m};\bm{n},j},
    \end{alignat}
    hold.
\end{prop}

In particular, in the special case when each $L_j$ corresponds to integration with respect to a positive measure on the unit circle~\eqref{eq:functionals given by measures} that is invariant under $e^{i\theta}\mapsto e^{-i\theta}$, then Propositions~\ref{thm:reversal vs star polynomials} and~\ref{rem:Rsymmetry} together imply that $\alpha_{\bm{n};\bm{m}} = \beta_{\bm{m};\bm{n}}$ and $\rho_{\bm{n};\bm{m}} = \sigma_{\bm{m};\bm{n}}$ are real.
Now we are ready to establish a relationship between multiple orthogonality on the real line and multiple orthogonality on the unit circle. For $r=1$ this relationship goes back to Szeg\H{o} \cite{Szego} (for functionals, it is effectively in~\cite{GeronimusRelations,GeronimusBook}, see also~\cite[Sect. 13]{OPUC2} and~\cite{DerSim18}). 

\begin{thm}\label{thm:szego mapping type II}
    Assume that $\bm{L} = (L_1,\dots,L_r)$ is a system of linear functionals on $\T$ satisfying~\eqref{eq:symmetricFunctionals} and let $\bm{M}$ be defined by $M_j=\operatorname{Sz}^{-1}(L_j)$, $j = 1,\dots,r$.

    If $(\bm{n};\bm{n})$ and $(\bm{n}+\bm{e}_j;\bm{n})$ are normal for $\bm{L}$ for all $\bm{n}\in\N^r$ and $j=1,\ldots,r$, then $\bm{M}$ is perfect (that is, every $\bm{n}\in\N^r$ is normal for $\bm{M}$), and
    \begin{align}\label{eq:SzegoPolynomialRelation}
          P_{\bm{n}}(z+z^{-1}) & = \frac{1}{(1+\alpha_{\bm{n};\bm{n}})} \Big( \Phi_{\bm{n};\bm{n}}(z) + \Phi_{\bm{n};\bm{n}}(1/z) \Big) 
          \\
          \label{eq:SzegoPolynomialRelation2}
          & = \Phi_{\bm{n};\bm{n}-\bm{e}_j}(z) + \Phi_{\bm{n};\bm{n}-\bm{e}_j}(1/z).
        \end{align}
        Furthermore, $\alpha_{\bm{n};\bm{n}}$ is necessarily $\ne -1$. 
\end{thm}
\begin{proof}
    The Laurent polynomial $f_{\bm{n}}(z)=\Phi_{\bm{n};\bm{n}}(z) + \Phi_{\bm{n};\bm{n}}(1/z)$ lies in $\operatorname{span}(z^{-|\bm{n}|},\ldots,z^{|\bm{n}|})$ and satisfies $f_{\bm{n}}(z) = f_{\bm{n}}(1/z)$. Therefore (see, e.g., ~\cite[Lem 13.1.4]{OPUC2}) there exists some polynomial $Q_{\bm{n}}(z)$ of degree at most $|\bm{n}|$ such that $Q_{\bm{n}}(z+z^{-1})= \Phi_{\bm{n};\bm{n}}(z) + \Phi_{\bm{n};\bm{n}}(1/z)$.

    Since each $L_j[w^k] = L_j[w^{-k}]$, 
    we are in the setting of Proposition~\ref{rem:Rsymmetry}, so that $\Phi_{\bm{n};\bm{n}}(1/z)=\Phi_{\bm{n};\bm{n}}^*(z)$. 
    Then 
    \begin{align}
    \label{eq:proofInner1}
     M_j[Q_{\bm{n}}(x)x^k]& =  L_j[ Q_{\bm{n}}(w+w^{-1})(w+w^{-1})^k]
    \\
    \label{eq:proofInner2}
    &= 
    L_j[ \Phi_{\bm{n};\bm{n}}(w)(w+w^{-1})^k] 
    +L_j[\Phi^*_{\bm{n};\bm{n}}(w)(w+w^{-1})^k] 
    \\
    \label{eq:proofInner3}
    & 
    =0, \qquad k = 0,\dots,n_j-1,
    \end{align}
    by ~\eqref{eq:moprlII}, and~\eqref{eq:moprlII*}.
    Hence $Q_{\bm{n}}$ satisfies the type II orthogonality conditions~\eqref{eq:MOPRLIIreal} for $\bm{M}$ at the location $\bm{n}$. 

    Now let us consider~\eqref{eq:proofInner1}--\eqref{eq:proofInner2} with $k=n_j$. We get 
    \begin{align}
     M_j[Q_{\bm{n}}(x)x^{n_j}]
    & 
    =L_j[ \Phi_{\bm{n};\bm{n}}(w) w^{-n_j}] 
    +L_j[\Phi_{\bm{n};\bm{n}}(1/w) w^{n_j}] 
    \\&\label{eq:computation for a}
    = 2L_j[ \Phi_{\bm{n};\bm{n}}(w) w^{-n_j}] ,
    \\&
    \label{eq:lastQnorm}
    = \frac{2\det T_{\bm{n}+\bm{e}_j;\bm{n}}}{\det T_{\bm{n};\bm{n}}},
    \end{align}
    where on the last two steps we used symmetry~\eqref{eq:symmetricFunctionals} and then~\eqref{eq:norm1}. By Proposition~\ref{prop:normality}{(i)} and normality of $(\bm{n}+\bm{e}_j;\bm{n})$, we obtain that the last expression is non-zero.
    Now we are in the position to apply the criterion  
    \cite[Thm 2.19]{KVChristoffel} 
    to conclude that $\bm{M}$ is perfect.  
    Note that the $z^{|\bm{n}|}$ coefficient of $Q_{\bm{n}}$ is equal to $1+\alpha_{\bm{n};\bm{n}}$. This proves that $1+\alpha_{\bm{n};\bm{n}}\ne 0$ and that ~\eqref{eq:SzegoPolynomialRelation} holds.  

    Finally, applying the same argument to $f(z)=\Phi_{\bm{n};\bm{n}-\bm{e}_j}(z) + \Phi_{\bm{n};\bm{n}-\bm{e}_j}(1/z)$ (note  that $(\bm{n};\bm{n}-\bm{e}_j)= ((\bm{n}-\bm{e}_j)+\bm{e}_j;\bm{n}-\bm{e}_j)$ is normal by assumption), one can see that ~\eqref{eq:SzegoPolynomialRelation} holds by repeating the argument in~\eqref{eq:proofInner1}--\eqref{eq:proofInner3} and noting that the $z^{|\bm{n}|}$ coefficient of $f$ is 1.
\end{proof}

In the next result we establish the generalization of the Geronimus relations that connect Verblunsky coefficients of $\mu$ and the Jacobi coefficients of $\gamma$, compare with the original~\cite{GeronimusRelations} Geronimus relations~\eqref{eq:GerIntro1}--\eqref{eq:GerIntro2}. 

\begin{thm}\label{thm:Geronimus}
    In the setting of the previous theorem, we have
    \begin{alignat}{1}
    \label{eq:Geronimus1}
        a_{\bm{n},j} & =
        \frac{(1+\alpha_{\bm{n}-\bm{e}_j;\bm{n}-\bm{e}_j})(1-\alpha_{\bm{n}-\bm{e}_j;\bm{n}}^2
        )\rho_{\bm{n};\bm{n},j} 
        }{1+\alpha_{\bm{n};\bm{n}}},
        \\
        \label{eq:Geronimus2}
        b_{\bm{n},j} & =
        \sum_{\ell=1}^r \rho_{\bm{n};\bm{n},\ell}  \gamma_{\bm{n};\bm{n}}^{\ell j}+
        \alpha_{\bm{n};\bm{n}-\bm{e}_j} - \alpha_{\bm{n}+\bm{e}_j;\bm{n}}
        - \alpha_{\bm{n};\bm{n}} \alpha_{\bm{n}-\bm{e}_j;\bm{n}}
        - \alpha_{\bm{n};\bm{n}} \alpha_{\bm{n}+\bm{e}_j;\bm{n}} .
    \end{alignat}
\end{thm}
\begin{rem}
    When $\alpha_{\bm{n};\bm{n}} \neq 0$ we can eliminate $\gamma_{\bm{n};\bm{n}}^{\ell j}$ to get
    \begin{equation}
        \label{eq:Geronimus2.1}
        b_{\bm{n},j}  =
        \frac{1}{\alpha_{\bm{n};\bm{n}}}\bigg(-\alpha_{\bm{n}+\bm{e}_j;\bm{n}}+\sum_{\ell=1}^r \alpha_{\bm{n}+\bm{e}_{\ell};\bm{n}}\rho_{\bm{n};\bm{n},\ell}\bigg)  +
        \alpha_{\bm{n},\bm{n}-\bm{e}_j} - \alpha_{\bm{n}+\bm{e}_j,\bm{n}}
        - \alpha_{\bm{n};\bm{n}} \alpha_{\bm{n}-\bm{e}_j;\bm{n}} .
    \end{equation}
\end{rem}
\begin{proof}
    From~\eqref{eq:one nnr} and~\eqref{eq:MOPRLIIreal}, and then from~\eqref{eq:type 1 recurrence relation} and~\eqref{eq:MOPRLIreal}--\eqref{eq:normalization type I moprl}, one gets
    \begin{align}
    \label{eq:avsM}
        a_{\bm{n},j} 
        &= \frac{ M_j[P_{\bm{n}}(x)x^{n_j} ]}{ M_j[P_{\bm{n}-\bm{e}_j}(x) x^{n_j - 1}] },
        \\
    \label{eq:bvsM}
        b_{\bm{n},j}
        & = k_{|\bm{n}|-1}(P_{\bm{n}}) - k_{|\bm{n}|}(P_{\bm{n}+\bm{e}_j}),
    \end{align}
    where $k_j(P)$ denotes the $z^j$ coefficient of a polynomial $P$.

    {Using ~\eqref{eq:avsM}, \eqref{eq:SzegoPolynomialRelation}, and \eqref{eq:computation for a}, we obtain
    \begin{equation}
        a_{\bm{n},j} = \frac{1+\alpha_{\bm{n}-\bm{e}_j;\bm{n}-\bm{e}_j}}{1+\alpha_{\bm{n};\bm{n}}}\frac{2L_j[\Phi_{\bm{n};\bm{n}}(z)z^{-n_j}]}{L_j[\Phi_{\bm{n}-\bm{e}_j;\bm{n}}(z)z^{-n_j+1}]}\frac{L_j[\Phi_{\bm{n}-\bm{e}_j;\bm{n}}(z)z^{-n_j+1}]}{2L_j[\Phi_{\bm{n}-\bm{e}_j;\bm{n}-\bm{e}_j}(z)z^{-n_j+1}]}.
    \end{equation}
    Now \eqref{eq:Geronimus1} follows from \eqref{eq:rho} and \eqref{eq:one minus alphabeta}, and the symmetry \eqref{eq:szego symmetry}.
    }

    For the second relation we follow the trick from the proof of~\cite[Thm 13.1.7]{OPUC2}. From~\eqref{eq:bvsM} and \eqref{eq:SzegoPolynomialRelation2}, we get
    \begin{align}
        b_{\bm{n},j} & =  (k_{|\bm{n}|-1}(\Phi_{\bm{n};\bm{n}-\bm{e}_j}) +\alpha_{\bm{n},\bm{n}-\bm{e}_j}) - (k_{|\bm{n}|}(\Phi_{\bm{n}+\bm{e}_j,\bm{n}})+\alpha_{\bm{n}+\bm{e}_j,\bm{n}})
        \\
        & =
        (k_{|\bm{n}|-1}(\Phi_{\bm{n};\bm{n}}) - k_{|\bm{n}|}(\Phi_{\bm{n}+\bm{e}_j,\bm{n}}) )
        +(\alpha_{\bm{n},\bm{n}-\bm{e}_j} - \alpha_{\bm{n}+\bm{e}_j,\bm{n}})
        - \alpha_{\bm{n};\bm{n}} \beta_{\bm{n};\bm{n}-\bm{e}_j},
    \end{align}
    where we used~\eqref{eq:SzegoII4} in the last line. The result now follows since we can compute $k_{|\bm{n}|-1}(\Phi_{\bm{n};\bm{n}}) - k_{|\bm{n}|}(\Phi_{\bm{n}+\bm{e}_j;\bm{n}})$ directly from \eqref{eq:SzegoII9} to get \eqref{eq:Geronimus2}, and then use symmetry \eqref{eq:szego symmetry}. For the remark, we instead compute $k_{|\bm{n}|-1}(\Phi_{\bm{n};\bm{n}}) - k_{|\bm{n}|}(\Phi_{\bm{n}+\bm{e}_j;\bm{n}})$ through \eqref{eq:generalized three term}.


\end{proof}

\bibsection




\begin{biblist}[\small]

\bib{Aptekarev}{article}{
   author={Aptekarev, A.I.},
   title={Multiple orthogonal polynomials},
   journal={J. Comput. Appl. Math.},
   volume={99},
   year={1998},
   pages={423--447},
}

\bib{Baker}{book}{
    AUTHOR = {Baker, G.A.},
    AUTHOR = {Graves-Morris, P.},
     TITLE = {Pad\'e{} approximants},
    SERIES = {Encyclopedia of Mathematics and its Applications},
    VOLUME = {59},
   EDITION = {Second},
 PUBLISHER = {Cambridge University Press, Cambridge},
      YEAR = {1996},
     PAGES = {xiv+746},
}


\bib{BCVA}{article}{
   author={Beckermann, B.},
   author={Coussement, J.},
   author={Van Assche, W.},
   title={Multiple Wilson and Jacobi-Pi\~neiro polynomials},
   journal={J. Approx. Theory},
   volume={132},
   date={2005},
   number={2},
   pages={155--181}
}

\bib{Bultheel}{book}{
    AUTHOR = {Bultheel, A.},
     TITLE = {Laurent series and their {P}ad\'e{} approximations},
    SERIES = {Operator Theory: Advances and Applications},
    VOLUME = {27},
 PUBLISHER = {Birkh\"auser Verlag, Basel},
      YEAR = {1987},
     PAGES = {xii+274},
}

\bib{Chihara}{book}{
   author={Chihara, T.S.},
   title={An Introduction to Orthogonal Polynomials},
   isbn={9780486479293},
   series={Mathematics and Its Applications},
   volume={13},
   publisher={Gordon and Breach Science Publishers, Inc.},
   year={1978},
}

\bib{MOPUC2}{article}{
   author={Cruz-Barroso, R.},
   author={Díaz Mendoza, C.},
   author={Orive, R.},
   title={Multiple orthogonal polynomials on the unit circle. Normality and recurrence relations},
   journal={J. Comput. Appl. Math.},
   volume={284},
   year={2015},
   pages={115--132},
}


\bib{DerSim18}{article}{
    AUTHOR = {Derevyagin, M.}
    AUTHOR = {Simanek, B.},
     TITLE = {Asymptotics for polynomials orthogonal in an indefinite
              metric},
   JOURNAL = {J. Math. Anal. Appl.},
  FJOURNAL = {Journal of Mathematical Analysis and Applications},
    VOLUME = {460},
      YEAR = {2018},
    NUMBER = {2},
     PAGES = {777--793},
}

\bib{GeronimusRelations}{article}{
    AUTHOR = {Geronimus, Ya.L.},
     TITLE = {On the trigonometric moment problem},
   JOURNAL = {Ann. of Math. (2)},
  FJOURNAL = {Annals of Mathematics. Second Series},
    VOLUME = {47},
      YEAR = {1946},
     PAGES = {742--761},
  MRNUMBER = {18265},
}

\bib{GeronimusBook}{book}{
    AUTHOR = {Geronimus, Ya.L.},
     TITLE = {Teoriya ortogonal\cprime nyh mnogo\v clenov},
 PUBLISHER = {Gosudarstv. Izdat. Tehn.-Teor. Lit., Moscow-Leningrad},
      YEAR = {1950},
     PAGES = {164},
}


\bib{HueMan}{article}{
    AUTHOR={Huertas, E.J.},
    AUTHOR={Ma\~{n}as, M.},
     TITLE = {Mixed-type multiple orthogonal {L}aurent polynomials on the
              unit circle},
   JOURNAL = {J. Comput. Appl. Math.},
  FJOURNAL = {Journal of Computational and Applied Mathematics},
    VOLUME = {475},
      YEAR = {2026},
     PAGES = {Paper No. 117037},
}


\bib{Ismail}{book}{
   author={Ismail, M.E.H.},
   title={Classical and Quantum Orthogonal
Polynomials in One Variable},
   isbn={9780521782012},
   series={Encyclopedia of Mathematics and its Applications},
   Volume={98},
   publisher={Cambridge University Press},
   year={2005},
}

\bib{JNT}{article}{
    author={Jones, W.B.},
    author={Njåstadt, O.},
    author={Thron, W.J.},
    title={Moment Theory, Orthogonal Polynomials, Quadrature, and Continued Fractions Associated with the unit Circle},
    journal={Bulletin of the London Mathematical Society},
    volume={21},
    number={2},
    year={1989},
    pages={113–152}
}

\bib{KNik}{article}{
    AUTHOR={Kozhan, R.},
    TITLE={Nikishin Systems on the Unit Circle},
  journal={},
   volume={},
   date={},
   number={},
   pages={under submission, arXiv:2410.20813},
}


\bib{KVMLOPUC1}{article}{
    AUTHOR={Kozhan, R.},
    AUTHOR={Vaktnäs, M.},
    TITLE={Angelesco and AT Systems on the Unit Circle},
    journal={},
    volume={},
    date={},
    number={},
    pages={under submission, arXiv:2410.12094},
}

\bib{KVChristoffel}{article}{
   author={Kozhan, R.},
   author={Vaktn\"{a}s, M.},
   title={Christoffel transform and multiple orthogonal polynomials},
      JOURNAL = {J. Comput. Appl. Math.},
    VOLUME = {476},
      YEAR = {2026},
     PAGES = {Paper No. 117121},
}

\bib{KVMOPUC}{article}{
    AUTHOR={Kozhan, R.},
    AUTHOR={Vaktnäs, M.},
    TITLE={Szeg\H{o} recurrence for multiple orthogonal polynomials on the unit circle},
    JOURNAL={Proc. Amer. Math. Soc.},
    VOLUME={152},
    NUMBER={11},
    YEAR={2024},
    PAGES={2983-2997},
    ISSN={1088-6826,0002-9939},
}

\bib{KVInterlacing}{article}{
    AUTHOR={Kozhan, R.},
    AUTHOR={Vaktnäs, M.},
    TITLE={Zeros of multiple orthogonal polynomials: location and interlacing},
    journal={Bull. of London Math. Soc.},
    volume={},
    date={},
    number={},
    pages={to appear, arXiv:2503.15122},
}


 
\bib{Kui}{article}{
   AUTHOR = {Kuijlaars, A.B.J.},
     TITLE = {Multiple orthogonal polynomial ensembles},
    JOURNAL = {Recent trends in orthogonal polynomials and approximation
              theory, Contemp. Math., Amer. Math. Soc., Providence, RI},
    VOLUME = {507},
     PAGES = {155--176},
      YEAR = {2010},
      ISBN = {978-0-8218-4803-6},
}

\bib{Applications}{article}{
   author={Martínez-Finkelshtein, A.},
   author={Van Assche, W.},
   title={WHAT IS...A Multiple Orthogonal Polynomial?},
   journal={Not. Am. Math. Soc.},
   volume={63},
   year={2016},
   pages={1029--1031},
}

\bib{MOPUC1}{article}{
   author={Mínguez Ceniceros, J.},
   author={Van Assche, W.},
   title={Multiple orthogonal polynomials on the unit circle},
   journal={Constr. Approx.},
   volume={28},
   year={2008},
   pages={173--197},
}

\bib{Nikishin}{book}{
    AUTHOR = {Nikishin, E.M.},
    AUTHOR = {Sorokin, V.N.},
     TITLE = {Rational approximations and orthogonality},
    SERIES = {Translations of Mathematical Monographs},
    VOLUME = {92},
      NOTE = {Translated from the Russian by Ralph P. Boas},
 PUBLISHER = {American Mathematical Society, Providence, RI},
      YEAR = {1991},
     PAGES = {viii+221},
      ISBN = {0-8218-4545-4},
}

\bib{PeherstorferSteinbauer}{article}{
    author={Peherstorfer, F.},
    author={Steinbauer, R.},
    title={Characterization of orthogonal polynomials with respect to a functional},
    journal={Journal of Computational and Applied Mathematics},
    volume={65},
    year={1995},
    pages={339-355},
}

\bib{OPUC1}{book}{
   author={Simon, B.},
   title={Orthogonal Polynomials on the Unit Circle, Part 1: Classical Theory},
   isbn={0-8218-3446-0},
   series={Colloquium Lectures},
   Volume={54},
   publisher={American Mathematical Society},
   year={2004},
}

\bib{OPUC2}{book}{
   author={Simon, B.},
   title={Orthogonal Polynomials on the Unit Circle, Part 2: Spectral Theory},
   isbn={978-0-8218-4864-7},
   series={Colloquium Lectures},
   Volume={54},
   publisher={American Mathematical Society},
   year={2005},
}


\bib{SimonL2}{book}{
   author={Simon, B.},
   title={Szeg\H{o}'s Theorem and Its Descendants: Spectral Theory for $L^2$ Perturbations of Orthogonal Polynomials},
   isbn={9780691147048},
   series={Porter Lectures},
   publisher={Princeton University Press},
   year={2011},
}

\bib{SzegoBook}{book}{
    AUTHOR = {Szeg\H o, G.},
     TITLE = {Orthogonal polynomials},
    SERIES = {American Mathematical Society Colloquium Publications},
    VOLUME = {Vol. XXIII, 4th ed},
 PUBLISHER = {American Mathematical Society, Providence, RI},
      YEAR = {1975},
     PAGES = {xiii+432},
}

\bib{Szego}{article}{
    author={Szeg\H{o}, G.},
    title={Über die Entwickelung einer analytischen Funktion nach den Polynomen eines Orthogonalsystems},
    issn={0025-5831},
    journal={Mathematische Annalen},
    volume={82},
    year={1921},
    pages={188-212},
}
	
\bib{NNRR}{article}{
   author={Van Assche, W.},
   title={Nearest neighbor recurrence relations for multiple
orthogonal polynomials},
   journal={J. Approx. Theory},
   volume={163},
   year={2011},
   pages={1427--1448},
}


\end{biblist}
\end{document}